\documentclass[a4paper,10pt]{amsart}

\usepackage[
  margin=30mm,
  marginparwidth=25mm,     
  marginparsep=2mm,       
  bottom=25mm,
  ]{geometry}

\usepackage[bbgreekl]{mathbbol}
\usepackage{amsfonts}
\usepackage{latexsym,amssymb,amsthm,mathrsfs,amsmath,amscd, amsfonts,amsthm,enumerate,enumitem,color}
\usepackage{bm}
\usepackage{verbatim}
\usepackage[T1]{fontenc}
\usepackage{calligra}
\usepackage{dsfont}
\usepackage{graphicx}
\usepackage{tikz}
\usepackage{cite}
\usepackage{cancel}

\DeclareSymbolFontAlphabet{\mathbb}{AMSb}
\DeclareSymbolFontAlphabet{\mathbbl}{bbold}
\DeclareFontEncoding{FMS}{}{}
\DeclareFontSubstitution{FMS}{futm}{m}{n}
\DeclareFontEncoding{FMX}{}{}
\DeclareFontSubstitution{FMX}{futm}{m}{n}
\DeclareSymbolFont{fouriersymbols}{FMS}{futm}{m}{n}
\DeclareSymbolFont{fourierlargesymbols}{FMX}{futm}{m}{n}
\DeclareMathDelimiter{\VERT}{\mathord}{fouriersymbols}{152}{fourierlargesymbols}{147}

\newtheorem{theorem}{Theorem}[section]
\newtheorem{corollary}[theorem]{Corollary}
\newtheorem{lemma}[theorem]{Lemma}
\newtheorem{proposition}[theorem]{Proposition}
\theoremstyle{definition}

\theoremstyle{remark}
\newtheorem{remark}[theorem]{Remark}

\usepackage{hyperref}
\hypersetup{
    colorlinks=true,
    linkcolor=blue,
    filecolor=magenta,
    urlcolor=cyan,
}

\newcommand{\N}{\mathbb{N}}

\newcommand{\C}{{\mathbb C}}

\newcommand{\R}{\mathbb{R}}
\newcommand{\Rm}{\R_+}
\newcommand{\Rn}{\R^n}
\newcommand{\Rnm}{\R^{n}_{+}}
\DeclareMathOperator*{\ceroinf}{(0,\infty)^{n}}

\newcommand{\supp}{\mathop{\mathrm{supp}}}

\newcommand{\esssup}{\mathop{\mathrm{ess\, sup \;}}}

\newcommand{\eps}{\varepsilon}

\newcommand{\HH}{H}
\newcommand{\nHH}{H}
\newcommand{\nBH}{\Delta_{\mu}}

\newcommand{\HZ}{h}
\newcommand{\nHZ}{h}

\newcommand{\nBZ}{S_{\mu}}

\newcommand{\DZ}{\textbf{D}}
\newcommand{\DHI}{\boldsymbol{\mathfrak{D}}}

\newcommand{\convZ}{\sharp}
\newcommand{\convH}{\#}

\let\Re\undefined

\DeclareMathOperator{\Re}{\mathrm{Re}}

\numberwithin{equation}{section}
\allowdisplaybreaks

\begin{document}


\title[]{$n$-dimensional Fractional Bessel Operators and Liouville theorems}

\author[V. Galli]{Vanesa Galli}
\author[S. Molina]{Sandra Molina}
\author[A. Quintero]{Alejandro Quintero}
\address{Vanesa Galli,
Sandra Molina,
Alejandro Quintero\newline
Departamento de Matem\'atica, Facultad de Ciencias Exactas y Naturales,\newline
Universidad Nacional de Mar del Plata, \newline
Funes 3350, 7600 Mar del Plata,
Argentina.}
\email{vggalli@mdp.edu.ar;
smolina@mdp.edu.ar;
aquinter@mdp.edu.ar
}



\date{}


\begin{abstract}
In this paper we extend the results given in \cite{Mo18} to the $n$-dimensional case the fractional powers of Bessel operators. Moreover, we established  a Liouville type theorems for these operators. This extend the result obtained in \cite{GMQ18} for Bessel operators.

\end{abstract}

\maketitle
\section{Introduction}
Bessel operators appear in the  setting of harmonic analysis related with Hankel transformations. In the one dimensional case, Bessel operators  appear when we consider the Laplacian operator in polar coordinates. In this work we study the fractional Bessel operator and Liouville theorems of  the $n$-dimensional versions in $\Rnm=\ceroinf$  given by
\begin{equation}\label{nOBZ}
S_{\mu}=
\sum_{i=1}^{n}\frac{\partial^{2}}{\partial x_{i}^{2}}+
\frac{4\mu_{i}^{2}-1}{4x_{i}^{2}}
\end{equation}
and
\begin{equation}\label{nOBH}
\nBH=
\sum_{i=1}^{n}
\frac{\partial^{2}}{\partial x_{i}^{2}}+
(2\mu_{i}+1)(x_{i}^{-1}\frac{\partial}{\partial x_{i}})
\end{equation}
where $\mu\in\Rn$, $\mu=(\mu_{1},\dots,\mu_{n})$ and $\mu_{i}>-\frac{1}{2}$.

In \cite{Mo18},  were studied the fractional powers of the one dimensional  case of \eqref{nOBZ} and \eqref{nOBH} in the sense of the classical theory of fractional powers developed by Balakrishnan in \cite{Ba60} and  using similarity of both operators. Let $X$ and $Y$ Banach spaces. Two linear operators $A$ and $B$, $A:D(A)\subset X\to X$ and $B:D(B)\subset Y\to Y$ are {\it similar} if there exists an  isomorphism $T:X\to Y$ with inverse  $T^{-1}:Y\to X$  such that  $D(B)=\{x\in Y: T^{-1}x\in D(A)\}$ given by
\begin{equation}\label{OpSimilares}
B=TAT^{-1}.
\end{equation}
Similar operators have the same spectral properties and also that of being non-negative if one of them has this property. Thus,  their powers are similar operators and verifies the same similarity relation, so $$ B^{\alpha}=TA^{\alpha}T^{-1}.$$
In this work, we generalize the results obtained in \cite{Mo18} to the $n$-dimensional case obtaining the fractional powers of Bessel operator \eqref{nOBZ} and \eqref{nOBH} in weighted Lebesgue spaces and in distributional spaces. As in \cite{Mo18} we first study the non-negativity  of Bessel operator \eqref{nOBZ} in suitable weighted Lebesgue spaces and by similarity we obtain the non-negativity of \eqref{nOBH} in the corresponding Lebesgue space.
Analogously to the one-dimensional case, we construct a locally convex space $\mathcal{B}$ in which $-S_{\mu}$ is continuous and non-negative. Next, we can consider the dual space $\mathcal{B'}$ with the strong topology and thus obtaining non-negativity of $-S_{\mu}$ in this distributional space. $\mathcal{B'}$ is contained in the distributional Zemanian space and  contain the   weighted Lebesgue spaces in which non-negativity  was studied. Consequently, if we denote with $(S_{\mu})_{\mathcal{B'}}$ the Bessel operator with domain  $\mathcal{B'}$,  we can consider the powers $(-(S_{\mu})_{\mathcal{B'}})^{\alpha}$ with $\Re(\alpha)>0$ and it is verified the following relation inherited from the selfadjuncture of $S_{\mu}$
$$((-S_{\mu})^{\alpha}u,\phi)=(u,(-S_{\mu})^{\alpha}\phi),$$
for $\phi\in\mathcal{B}$ and $u\in\mathcal{B'}$.

In \cite{GMQ18},  a Liouville-type theorem was studied for a certain general class of Bessel-type operators. This class of operators contain as a particular case the Bessel operator \eqref{nOBZ}, and the Liouville theorem applied to this operator states that if $u$ is a Zemanian distribution that verifies that $S_{\mu}u=0$ then $u$ is a polynomial. This result is analogous to that existing for the Laplacian operator and temperate distributions that states that any harmonic tempered distribution  is a polynomial. In \cite{BKN02}, \cite{Li06}, \cite{ZCCY14} and \cite{CDAL14} different version of Liouville theorem for the fractional Laplacian are studied.  In this paper we prove the following Liouville theorems for the distributional  fractional Bessel operators:

\begin{theorem}\label{Teorema5.1}
Let $u\in\mathcal{B'}$ and $\alpha\in\C$ with $\Re\alpha>0$. If $(-(S_{\mu})_{\mathcal{B'}})^{\alpha}u=0$ then there exists a polynomial $p$ such that $u=x^{\mu+\frac{1}{2}}p(\|x\|^{2})$.
\end{theorem}

For  the study of the powers of Bessel operator given by \eqref{nOBH} we introduce a locally convex space $\mathcal{F}$. This space verifies that its dual space $\mathcal{F'}$ with the strong topology is a suitable distributional space for the study of fractional powers $(-(\Delta_{\mu})_{\mathcal{F'}})^{\alpha}$ and from similarity we conclude the following result

\begin{theorem} \label{teorema 8.2}
Let $u\in\mathcal{F'}$ and $\alpha\in\C$ with $\Re\alpha>0$. If $(-(\Delta_{\mu})_{\mathcal{F'}})^{\alpha}u=0$ then there exists a polynomial $p$ such that $u=x^{2\mu+1}p(\|x\|^{2})$.
\end{theorem}

This paper is organized as follow. In section 2 we summarize basic results related with harmonic analysis in the Hankel setting. Section 3 contain a brief review of non-negative operators in Banach and locally convex spaces and properties of fractional powers of similar operators. In sections 4, 5 and 6 we study the non-negativity of Bessel operator \eqref{nOBZ} and \eqref{nOBH} . Finally, sections 6 and 7 contains Liouville's theorems for the two fractional Bessel operators.

\section{Preliminaries}
In this section we introduce the Lebesgue and distributional spaces necessary four our purposes.

We now present some notational conventions that will allow us to simplify the presentation of our results. Let $\Rn$ be the $n$-dimensional  euclidean space, $\Rnm=\ceroinf$ the $n$-tuples of real positive numbers. We denote by  $x=(x_1,\ldots,x_n)$ and $y=(y_1,\ldots,y_n)$ to the elements of $\ceroinf$ or $\R^{n}$ and let $\N$ be the set $\{1,2,3,\ldots\}$ and $\N_{0}=\N\cup\{0\}$. For $x\in\Rn$, the norm is given by $\|x\|=(x_{1}^{2}+\ldots + x_{n}^{2})^{\frac{1}{2}}$. The notations $x<y$ and   $x\leq y$ mean, respectively $x_i<y_i$ and  $x_i\leq y_i$, for $i=1,\ldots,n$. Moreover if  $x\in\R^{n}$ and $a\in\R$, $x=a$ means $x_1=x_2=\ldots=x_n=a$. We denote $e_j$, for  $j=1,\ldots,n$, the elements of the canonical basis $\R^{n}$.
An element $k=(k_1,\ldots,k_n)\in\N_{0}^{n}=\N_{0}\times\N_{0}\times\ldots\times\N_{0}$ is called multi-index. For $k,m$ multi-index we set $|k|=k_1+\ldots+k_n$.

Also we will note
\[
k!=k_1!\ldots k_n!
\qquad \text{and}
\qquad \binom{k}{m}=\binom{k_1}{m_1}\ldots
\binom{k_n}{m_n}
\qquad \text{for} \: k,m\in\N_{0}^{n}.
\]
If $x\in\Rn$ and $\beta\in\Rn$, we define
\begin{equation}\label{x^b}
x^{\beta}=x_{1}^{\beta_1}\ldots x_{n}^{\beta_n}.
\end{equation}
In particular if $a\in\R$, $a^{\beta}$ means
\begin{equation}\label{constante^b}
a^{\beta}=a^{\beta_1}\ldots a^{\beta_n},
\end{equation}
and if  $\beta$ is a multi-index, $\beta\in\N_{0}^{n}$
\[a^{\beta}=a^{|\beta|}.\]
For  $\alpha\in\R$ let $\bm{\alpha}=(\alpha,\ldots,\alpha)$,
then for  $a\in\R\:\text{and}\:x\in\Rn$
\begin{equation}\label{x^{constante}}
a^{\bm{\alpha}}=(a^{n})^{\alpha}
\quad \text{y}\quad
x^{\bm{\alpha}}=x_{1}^{\alpha}\ldots x_{n}^{\alpha}=x^{\alpha}.
\end{equation}

\noindent If $\beta$ is a multi-index, $\beta=(\beta_{1},\ldots,\beta_{n})$ y $\alpha\in\R$ let
\begin{equation}\label{multi-real}
\beta-\alpha=
(\beta_{1}-\alpha,\ldots,\beta_{n}-\alpha)=
\beta-\bm{\alpha}
\end{equation}

\bigskip
\noindent If $D_j=\frac{\partial}{\partial x_j}$, $j=1,\ldots,n$ then the partial derivatives respect to $x$ is denoted by
\[D^{k}=D_{1}^{k_1}\ldots D_{n}^{k_n}.\]
where $k$ is a multi-index. We define the operators \[T_{j}=x_{j}^{-1}D_{j}\]
for  $j=1,\ldots,n$. For a multi-index $k$ we shall write
\[T^{k}=T_{n}^{k_n}\circ T_{n-1}^{k_{n-1}}\circ\ldots\circ T_{1}^{k_1}.\]

\begin{remark}
Let $k$ be a multi-index and $\theta,\varphi$ differentiable functions up to order $|k|$, the following equality is valid
\begin{equation}\label{Leibniz}
  T^{k}\{\theta\cdot\varphi\}=\sum_{j=0}^{k}\binom{k}{j}T^{k-j}\theta. T^{j}\varphi,
\end{equation}

\noindent where $\:"\cdot\:"$  denote the usual product of functions.
\end{remark}

\bigskip
Hankel transformation appears in mathematical literature in various forms,  two classical versions correspond to the versions studied by A. H. Zemanian \cite{Ze87,Ze66}
\begin{equation}\label{1HZ}
(\HZ_{\alpha}f)(t)=
\int_{0}^{\infty}
f(x)\sqrt{xt}J_{\alpha}(xt)\:dx,
\qquad t\in(0,\infty)
\end{equation}
and I. I. Hirschman \cite{Hi60}
\begin{equation}\label{1HH}
(\HH_\alpha f)(t)=
\int_{0}^{\infty} f(x)(xt)^{-\alpha}J_{\alpha}(xt)\:x^{2\alpha+1}\:dx,
\qquad t\in(0,\infty)
\end{equation}
where $\alpha>-\frac{1}{2}$ and $J_{\alpha}$ is the well known Bessel function of first kind and order $\alpha$.

\medskip

S. Molina y S. Trione studied in \cite{MT07,MT08} a $n$-dimensional  generalization  of \eqref{1HZ}, given by $\nHZ_{\mu}$ and defined by
\begin{equation}\label{nHZ}
(\nHZ_{\mu}\phi)(y)=
\int_{\Rnm}\phi(x_{1},\ldots,x_{n})\prod_{i=1}^{n}
\{\sqrt{x_{i}y_{i}}J_{\mu_{i}}(x_{i}y_{i})\}
\;dx_{1}\ldots dx_{n}.
\end{equation}

Analogously it is possible to define a $n$-dimensional generalization for
\eqref{1HH}, given by  $\nHH_{\mu}$ and defined by
\begin{equation}\label{nHH}
(\nHH_{\mu}\phi)(y)=
\int_{\Rnm}\phi(x_1,\ldots,x_n)\left\{\prod\limits_{i=1}^{n}
(x_iy_i)^{-\mu_i}J_{\mu_i}(x_iy_i)\,x_i^{2\mu_i+1}\right\}
\: dx_1\ldots dx_n.
\end{equation}

\noindent In both, \eqref{nHZ} and  \eqref{nHH},  $\mu=(\mu_{1},\ldots,\mu_{n})$, $\mu_{i}>-\frac{1}{2}$ and   $J_{\mu_i}$ represents the Bessel function of first kind and order $\mu_i$ for $i=1,\ldots n$.

\bigskip

Next we define certain weighted $L^{p}$-spaces for $1\leq p\leq\infty$. Let
\begin{equation}\label{peso-s}
s(x)=\frac{x^{2\mu+1}}{C_\mu}
\end{equation}
\begin{equation}\label{peso-r}
r(x)=x^{-\mu-\frac{1}{2}}
\end{equation}

 where $\mu=(\mu_{1},\ldots,\mu_{n})$, $x\in\Rnm$, $C_\mu=2^\mu\,\Gamma(\mu_1+1)\ldots \Gamma(\mu_n+1)$ and $dx$ is the usual $n$-dimensional Lebesgue and the powers  $x^{2\mu+1}$ and  $x^{-\mu-1/2}$ are given by \eqref{x^b} and $2^{\mu}$ is given by \eqref{constante^b}. Let  $L^{p}(\Rnm,sr^p)$, $1\leq p<\infty$ the space of measurable functions $f$  defined over $\Rnm$ with norm

\[\|f\|_{L^{p}(\Rnm,sr^p)}=
\left(\int_{\Rnm} |f(x)|^{p}\,s(x)r^p(x)\:dx\right)^{1/p}  \quad 1\leq p<\infty.\]

Moreover,  $L^{\infty}(\Rnm,r)$, is the space of measurable functions over  $\Rnm$ such that
 \[\|f\|_{L^{\infty}(\Rnm,r)}= \esssup_{x\in\Rnm}\:|r(x)f(x)|<\infty.\]

In particular, if $p=2$, $L^{2}(\Rnm,sr^2)=L^{2}(\Rnm)$.

\bigskip
For simplicity sometimes we write $L^{p}(sr^p)$ and $L^{\infty}(r)$ instead of $L^{p}(\Rnm,sr^p)$ and $L^{\infty}(\Rnm,r)$.

\medskip

By $\mathcal{D}(\Rnm)$ we denote the space of functions in $C^{\infty}(\Rnm)$ with compact support in $\Rnm$ with the usual topology, and by $\mathcal{D'}(\Rnm)$ the space of classical distributions in $\Rnm$.

\bigskip

We consider the Zemanian space $\mathcal{H}_{\mu}$  of the functions $\phi\in C^{\infty}(\Rnm)$ such that
\[
\sup_{x\in\Rnm}
|(1+\|x\|^{2})^{m}T^{k}\{x^{-\mu-1/2}\phi(x)\}|<\infty:
m\in\N_{0},\;k\in\N_{0}^{n}
\]

\noindent endowed with the topology generated by the family of seminorms $\{\nu_{m,k}^{\mu}\}$, given by
\begin{equation}\label{Topologia H_mu}
\nu_{m,k}^{\mu}(\phi)=
\sup_{x\in\Rnm}
|(1+\|x\|^{2})^{m}T^{k}\{x^{-\mu-1/2}\phi(x)\}|
\end{equation}

\noindent where $-\mu-1/2=(-\mu_1-1/2,\ldots,-\mu_n-1/2)$ and the operators $T^{k}$ are given by
\[
T^{k}=T_{n}^{k_n}\circ T_{n-1}^{k_{n-1}}\circ\ldots\circ T_{1}^{k_1},
\]
\noindent where $T_{i}=x_{i}^{-1}\frac{\partial}{\partial x_{i}}$ and $k=(k_1,\ldots,k_n)$. $\mathcal{H}_{\mu}$ is a Fréchet space (see \cite{MT07}). The dual space of $\mathcal{H}_\mu$ is denoted by $\mathcal{H'}_\mu$.

\begin{remark}
Sometimes we will consider the family of seminorms
\begin{equation}\label{Topologia H_mu-2}
\gamma_{m,k}^{\mu}(\phi)=
\sup_{x\in\Rnm}
|x^{m}T^{k}\{x^{-\mu-1/2}\phi(x)\}|
\end{equation}
with $m,k\in\N_{0}^{n}$, which are equivalent to $\nu_{m,k}^{\mu}$.
\end{remark}

\begin{lemma}\label{lema3.1}
The following inclusions hold
\begin{equation}\label{lema3.1-eq1}
\mathcal{H}_{\mu}\subset
L^{1}(\Rnm,sr)\cap L^{\infty}(\Rnm,r)\subset
L^{p}(\Rnm,sr^{p}),
\quad 1\leq p<\infty
\end{equation}
where $s$ and $r$ are given by  \eqref{peso-s} and  \eqref{peso-r}
respectively.
\end{lemma}

\begin{proof}
Let $\phi\in\mathcal{H}_{\mu}$,

\begin{equation}\label{lema3.1-eq2}
\|\phi\|_{L^{\infty}(\Rnm,r)}=
\sup\limits_{x\in\Rnm}|x^{-\mu-1/2}\phi(x)|=
\gamma_{0,0}^{\mu}(\phi),
\end{equation}
then $\phi\in L^{\infty}(\Rnm,r)$.

\bigskip

\noindent To show that $\mathcal{H}_{\mu}\subset
L^{1}(\Rnm,sr)$, let $\phi\in\mathcal{H}_{\mu}$, $m\in\N$ such that $m>2\mu_{i}+2$, for $i=1,\ldots,n$,  then
\begin{align*}
\int_{\Rnm}|\phi(x)|s(x)r(x)\;dx
&=
\int_{(0,1]^{n}}
|x^{-\mu-1/2}\phi(x)|\;\frac{x^{2\mu+1}}{C_{\mu}} dx +
\int_{\Rnm-(0,1]^{n}}
x^{m} |x^{-\mu-1/2}\phi(x)|\;\frac{x^{2\mu+1-m}}{C_{\mu}} dx\\
& \leq
\gamma_{0,0}^{\mu}(\phi)\:C_{\mu}^{-1}
\int_{(0,1]^{n}}
x^{2\mu+1} dx +
\gamma_{m,0}^{\mu}(\phi)\:C_{\mu}^{-1}
\int_{\Rnm-(0,1]^{n}}
x^{2\mu+1-m} dx
<\infty.
\end{align*}

\noindent Thus

\begin{equation}\label{lema3.1-eq3}
\|\phi\|_{L^{1}(\Rnm,sr)}\leq
C\{\gamma_{0,0}^{\mu}(\phi)+\gamma_{m,0}^{\mu}(\phi)\},
\quad \phi\in\mathcal{H}_{\mu}.
\end{equation}

\bigskip

\noindent Now let us see that $L^{1}(\Rnm,sr)\cap L^{\infty}(\Rnm,r)\subset  L^{p}(\Rnm,sr^{p})$. Let $\phi\in L^{1}(\Rnm,sr)\cap L^{\infty}(\Rnm,r)$

\begin{align*}
\int_{\Rnm} |\phi(x)|^{p}& s(x)r^{p}(x)\;dx
=
\int_{\Rnm} |\phi(x)|^{p-1}r(x)^{p-1}\;|\phi(x)|s(x)r(x)\;dx\\
&=
\int_{\Rnm} |r(x)\phi(x)|^{p-1}\;|\phi(x)|s(x)r(x)\;dx\\
&\leq
\|\phi\|_{L^{\infty}(\Rnm,r)}^{p-1}
\|\phi\|_{L^{1}(\Rnm,sr)},
\end{align*}

\noindent from where

\begin{equation}\label{lema3.1-eq4}
\|\phi\|_{L^{p}(\Rnm,sr^{p})}\leq
\|\phi\|_{L^{\infty}(\Rnm,r)}^{\frac{p-1}{p}}
\|\phi\|_{L^{1}(\Rnm,sr)}^{\frac{1}{p}}
\end{equation}

From \eqref{lema3.1-eq2} and  \eqref{lema3.1-eq3} we can consider that there exist constants $C_{1}$ y $C_{2}$ such that

\begin{equation}\label{lema3.1-eq5}
\|\phi\|_{L^{\infty}(\Rnm,r)}\leq
C_{1}\{\gamma_{0,0}^{\mu}(\phi)+\gamma_{m,0}^{\mu}(\phi)\},
\quad \phi\in\mathcal{H}_{\mu}.
\end{equation}

\begin{equation}\label{lema3.1-eq6}
\|\phi\|_{L^{1}(\Rnm,sr)}\leq
C_{2}\{\gamma_{0,0}^{\mu}(\phi)+\gamma_{m,0}^{\mu}(\phi)\},
\quad \phi\in\mathcal{H}_{\mu}.
\end{equation}

Then from \eqref{lema3.1-eq4},  \eqref{lema3.1-eq5} y  \eqref{lema3.1-eq6} we can consider a constant $C_{3}$ such that

\begin{equation}\label{lema3.1-eq7}
\|\phi\|_{L^{p}(\Rnm,sr^{p})}\leq
C_{3}\{\gamma_{0,0}^{\mu}(\phi)+\gamma_{m,0}^{\mu}(\phi)\},
\quad \phi\in\mathcal{H}_{\mu}.
\end{equation}
\end{proof}

\bigskip

\begin{remark}
If $\phi\in L^{1}(\Rnm,sr)$, then Hankel transform $\nHZ_{\mu}\phi$ is well defined because the kernel $(x_iy_i)^{-\mu_i}J_{\mu_i}(x_iy_i)$ is bounded for  $\mu_{i}>-\frac{1}{2}$,  $i=1,\ldots,n$ (see \cite[(1), pp.49]{Wa}),
\begin{align*}
&\int_{\Rnm}|\phi(x)|
\prod_{i=1}^{n}\{(x_iy_i)^{\mu_{i}+1/2}
|(x_iy_i)^{-\mu_{i}} J_{\mu_i}(x_iy_i)|\}\;dx\\
&\leq
y^{\mu+1/2} M^{n}\int_{\Rnm} |\phi(x)|\;x^{\mu+1/2}\;dx
= C y^{\mu+1/2} \|\phi\|_{L^{1}(\Rnm,sr)}<\infty
\end{align*}

\bigskip

By Lemma \ref{lema3.1}, $\nHZ_{\mu}\phi$ is well defined for all $\phi\in\mathcal{H}_{\mu}$ and is an automorphism of $\mathcal{H}_{\mu}$ (see \cite{Ze87} for the $1$-dimensional  case and \cite{MT07} for the $n$-dimensional case.)
\end{remark}

\bigskip
The space of continuous linear functions $T:\mathcal{H}_{\mu}\to \C$ is denoted by $\mathcal{H'}_{\mu}$. We call a function  $f\in L^{1}_{loc}(\Rnm)$ a regular element of $\mathcal{H'}_{\mu}$ if the application $T_{f}\in \mathcal{H'}_{\mu}$ where $T_{f}(\phi)=\int_{\Rnm}f\phi$, with $\phi\in\mathcal{H}_{\mu}$.

\begin{lemma}\label{lema3.2}
Let $1\leq p <\infty$. A function in $L^{p}(\Rnm,sr^{p})$ or in $L^{\infty}(\Rnm,r)$ is a regular element of $\mathcal{H'}_{\mu}$. In particular, the functions in $\mathcal{H}_{\mu}$ can be considered as regular elements of $\mathcal{H'}_{\mu}$.
\end{lemma}

\begin{proof}
Let $f\in L^{\infty}(\Rnm,r)$ and  $\phi\in\mathcal{H}_{\mu}$. Since $\mathcal{H}_{\mu}\subset L^{1}(\Rnm,sr)$, then $\phi\in L^{1}(\Rnm,r^{-1})$ and $(T_{f},\phi)=\int_{\Rnm}f\phi$ is well defined. So, by \eqref{lema3.1-eq3}

\begin{align*}
|(T_{f},\phi)|
&\leq
\|f\|_{L^{\infty}(\Rnm,r)}
\|\phi\|_{L^{1}(\Rnm,r^{-1})}=
C_{\mu}
\|f\|_{L^{\infty}(\Rnm,r)}
\|\phi\|_{L^{1}(\Rnm,sr)}\\
&\leq C\;C_{\mu}\|f\|_{L^{\infty}(\Rnm,r)}
\{\gamma_{0,0}^{\mu}(\phi)+\gamma_{m,0}^{\mu}(\phi)\},
\end{align*}

\noindent consequently, $f$ is a regular element of $\mathcal{H'}_{\mu}$.

\medskip

Now, let $f\in L^{p}(\Rnm,sr^{p})$ with $1\leq p<\infty$ and $\phi\in\mathcal{H}_{\mu}$, then

\begin{equation}\label{lema3.2-eq1}
\begin{split}
|(T_{f},\phi)|
&\leq
\int_{\Rnm}|f(x)\phi(x)|\;dx=
\int_{\Rnm} |r(x)f(x)|\;|s^{-1}(x)r^{-1}(x)\phi(x)|\;s(x)\;dx\\
&=
\int_{\Rnm} |r(x)f(x)|\;C_{\mu}|r(x)\phi(x)|\;s(x)\;dx.
\end{split}
\end{equation}

Since $r|f|\in L^{p}(\Rnm,s)$ and $r|\phi|\in L^{q}(\Rnm,s)$, being  $q$  such that $\frac{1}{p} + \frac{1}{q} = 1$, then due to H\"{o}lder's inequality and  \eqref{lema3.1-eq7} we obtain that
\[
|(T_{f},\phi)|
\leq
C_{\mu} \|f\|_{L^{p}(\Rnm,sr^{p})}
        \|\phi\|_{L^{q}(\Rnm,sr^{q})}
\leq
C\;C_{\mu} \|f\|_{L^{p}(\Rnm,sr^{p})}
        \{\gamma_{0,0}^{\mu}(\phi)+\gamma_{m,0}^{\mu}(\phi)\}
\]
with $m>2\mu_{i}+2$, for $i=1,\ldots,n$. Therefore $f$ is a regular element of $\mathcal{H'}_{\mu}$.
\end{proof}

\begin{remark}\label{L^2 subset H'mu}
In particular if  $p=2$, $L^{p}(\Rnm,sr^{p})=L^{2}(\Rnm)$ and from the previous  Lemma  we have that the functions in $L^{2}(\Rnm)$ can be considered as regular elements of $\mathcal{H'}_{\mu}$.
\end{remark}

\bigskip

Given $f,g$ defined on $\Rnm$, the Hankel convolution associated to the transformation  $\nHZ_{\mu}$ is defined formally by

\begin{equation}\label{nconv_Z}
(f\convZ g)(x)=
\int_{\Rnm}\int_{\Rnm}
\DZ_{\mu}(x,y,z)f(y)g(z)\;dy\;dz
\end{equation}
where for every $x,y,z\in\Rnm$,
\begin{equation}\label{nDZ} 
\DZ_{\mu}(x,y,z)=
\prod_{i=1}^{n} D_{\mu_i}(x_i,y_i,z_i)
\end{equation}
\noindent where  $D_{\alpha}$  is the Delsarte kernel defined in \cite{Delsarte}, given by

\begin{equation}\label{DZ}
D_{\alpha}(u,v,w)=
\frac{2^{\alpha-1}\;(uvw)^{-\alpha+1/2}}
{\Gamma(\alpha+1/2)\sqrt{\pi}}
A(u,v,w)^{2\alpha-1}
\end{equation}
\noindent and  $A(u,v,w)$ is the area of the triangle with sides $u,v,w\in\Rm$ y $\alpha\in\R$, $\alpha>-\frac{1}{2}$.

\medskip

Note that $|u-v|<w<u+v$ is the condition for such triangle to exist, and in this case
\begin{equation}\label{Área-lados}
A(u,v,w)=
\begin{cases}
\frac{1}{4}\sqrt{[(u+v)^{2}-w^{2}][w^{2}-(u-v)^{2}]}
&|u-v|<w<u+v\\
\hfil 0
& 0<w<|u-v|\quad\text{or}\quad w>u+v,
\end{cases}
\end{equation}

\begin{remark}
If $u,v$ and $w$ are the sides of a triangle and $\theta$ is the angle opposite the side $w$, then
\[A(u,v,w)= \frac{uv\sin\theta}{2}\]
\end{remark}

\begin{proposition}\label{prop-nDZ}
\begin{itemize}
\item[]
\item[(i)] $\DZ_{\mu}(x,y,z)\geq 0,\quad x,y,z\in\Rnm$.
\item[(ii)]$\int_{\Rnm}\DZ_{\mu}(x,y,z)
            \prod\limits_{i=1}^{n}
            \{\sqrt{z_it_i}J_{\mu_i}(z_it_i)\}\;dz=
            t^{-\mu-1/2}
             \prod\limits_{i=1}^{n}
            \{\sqrt{x_it_i}J_{\mu_i}(x_it_i)\}
             \prod\limits_{i=1}^{n}
            \{\sqrt{y_it_i}J_{\mu_i}(y_it_i)\}$
\item[(iii)] $\int_{\Rnm}z^{\mu+1/2}\DZ_{\mu}(x,y,z)\;dz=
             C_{\mu}^{-1} x^{\mu+1/2}y^{\mu+1/2}$
\end{itemize}
\end{proposition}

\begin{proof}
 For the proof of this result, we refer the reader to the Appendix, page \pageref{Proof-prop-nDZ}.
\end{proof}

\begin{lemma}\label{lema3.3}
Let  $f\in L^{1}(\Rnm,sr)$.
\begin{itemize}
\item[(i)] If $g\in L^{\infty}(\Rnm,r)$, then the convolution $f\convZ g(x)$ exists for every $x\in\Rnm$,  $f\convZ g(x)\in L^{\infty}(\Rnm,r)$ and
\begin{equation}\label{lema3.3-eq1}
\|f\convZ g\|_{L^{\infty}(\Rnm,r)}
\leq
\|f\|_{L^{1}(\Rnm,sr)}\|g\|_{L^{\infty}(\Rnm,r)}.
\end{equation}

\item[(ii)] If $g\in L^{p}(\Rnm,sr^{p})$, $1\leq p<\infty$, then the convolution $f\convZ g(x)$ exists for almost every $x\in\Rnm$,  $f\convZ g(x)\in L^{p}(\Rnm,sr^{p})$ and
\begin{equation}\label{lema3.3-eq2}
\|f\convZ g\|_{L^{p}(\Rnm,sr^{p})}
\leq
\|f\|_{L^{1}(\Rnm,sr)}\|g\|_{L^{p}(\Rnm,sr^{p})}.
\end{equation}
\end{itemize}
\end{lemma}
\begin{proof}
 For the proof of Lemma \ref{lema3.3}, refer to the Appendix, page \pageref{Proof-lema3.3}.
\end{proof}

\vspace{.2in}
The proof of the following results uses standard arguments and it will be omitted.
\begin{lemma}\label{lema3.4}
Let $f,g\in L^{1}(\Rnm,sr)$, then
\begin{equation}\label{lema3.4-eq1}
\nHZ_{\mu}(f\convZ g) = r\nHZ_{\mu}(f) \nHZ_{\mu}(g).
\end{equation}
\end{lemma}

\begin{lemma}
Let $f\in L^{1}(sr)$, then the Hankel transform $\nHZ_{\mu}f\in L^{\infty}(r)$ and \[\|\nHZ_{\mu}f\|_{L^{\infty}(r)}\leq \|f\|_{L^{1}(sr)}.\]
\end{lemma}

\begin{remark}\label{remark7.4 [Mo18]}
Given $f\in L^{1}(\Rnm)$ we have that $\nHZ_{\mu}f$ is continuous, is in  $L^{\infty}(\Rnm)$ and
\[
\|\nHZ_{\mu}f\|_{\infty}\leq C\|f\|_{1}.
\]
\end{remark}

\begin{proposition}\label{prop7.5 [Mo18]}
$\nHZ_{\mu}(L^{1}(\Rnm))\subset C_{0}(\Rnm)$
\end{proposition}
\begin{proof}
First, we observe that
\begin{equation}\label{prop7.5 [Mo18]-eq1}
L^{1}(\Rnm,sr)\cap L^{\infty}(\Rnm,r)
\subset
L^{1}(\Rnm).
\end{equation}
Let $Q$ the cube  $Q=[0,1]^{n}$, then
\begin{align*}
\int_{\Rnm}|f(x)|\:dx
&=
\int_{\Rnm}|f(x)|\:r(x)\:r^{-1}(x)\:dx\\
&=
\int_{Q\cap\Rnm}|f(x)|\:r(x)\:r^{-1}(x)\:dx
+
\int_{Q^{c}\cap\Rnm}|f(x)|\:r(x)\:r^{-1}(x)\:dx\\
&\leq
\|f\|_{L^{\infty}(\Rnm,r)}
\int_{Q\cap\Rnm} r^{-1}(x)\:dx
+
\int_{Q^{c}\cap\Rnm}|f(x)|\:r^{-1}(x)\:dx\\
&\leq
C\|f\|_{L^{\infty}(\Rnm,r)}
+
C_{\mu}\|f\|_{L^{1}(\Rnm,sr)},
\end{align*}
because $r(x)<1$ for $\|x\|>1$, $\mu_{i}>-\frac{1}{2}$, $i=1,\ldots,n$ and $r(x)s(x)=C_{\mu}^{-1}r^{-1}(x)$.

\bigskip
By \eqref{lema3.1-eq1} and \eqref{prop7.5 [Mo18]-eq1} we deduce that $\mathcal{H}_{\mu}\subset L^{1}(\Rnm)$. Since  $\mathcal{D}(\Rnm)\subset\mathcal{H}_{\mu}$ then $\mathcal{H}_{\mu}$ is dense in $L^{1}(\Rnm)$. Given $f\in L^{1}(\Rnm)$ and  $\{\phi_{m}\}\in \mathcal{H}_{\mu}$ such that $\phi_{m}\to f$ in $L^{1}(\Rnm)$, then by Remark \ref{remark7.4 [Mo18]} $\nHZ_{\mu}(\phi_{m})\to\nHZ_{\mu}(f)$ uniformly. Since $\nHZ_{\mu}(\phi_{m})\in C_{0}(\Rnm)$ then $\nHZ_{\mu}(f)\in C_{0}(\Rnm)$.

\smallskip
\end{proof}

\clearpage
\begin{lemma}\label{lema3.5}
Let $\{\phi_{m}\}\subset L^{1}(\Rnm,sr)$ such that
\begin{enumerate}
\item[(1)] $\phi_{m}\geq 0$ in $\Rnm$,
\item[(2)] $\int_{\Rnm} \phi_{m}(x)\;s(x)r(x)\;dx=1$
           for all $m\in\N$,
\item[(3)] For all $\eta>0$,
$\lim\limits_{m\to\infty}
\int_{\|x\|>\eta} \phi_{m}(x)\;r(x)s(x)\;dx=0$.
\end{enumerate}

\noindent Let $f\in L^{\infty}(\Rnm,r)$ and continuous in  $x_{0}\in\Rnm$, then $\lim\limits_{m\to\infty} f\convZ \phi_{m}(x_{0}) = f(x_{0})$. Moreover, if $rf$ is uniformly continuous in $\Rnm$ then \[\lim\limits_{m\to\infty} \|f\convZ\phi_{m}(x)-f(x)\|_{L^{\infty}(\Rnm,r)} = 0.\]
\end{lemma}

\begin{proof}
First let us observe that

\begin{align}\label{lema3.5-eq1}
\int_{\Rnm}&
\int_{\Rnm}
x_{0}^{-\mu-1/2} y^{\mu+1/2}
\DZ_{\mu}(x_{0},y,z)\:\phi_{m}(z)\:dy
dz\\
&=
\int_{\Rnm}
x_{0}^{-\mu-1/2} \phi_{m}(z)
\left\{\int_{\Rnm}
y^{\mu+1/2}
\DZ_{\mu}(x_{0},y,z)\:dy
\right\}dz\nonumber\\
&=
\int_{\Rnm}
x_{0}^{-\mu-1/2} \phi_{m}(z)
C_{\mu}^{-1} x_{0}^{\mu+1/2} z^{\mu+1/2}\:dz \nonumber\\
&=
\int_{\Rnm}
\phi_{m}(z)\:s(z)r(z)\:dz = 1.\nonumber
\end{align}

\bigskip

\noindent Let $\eps>0$, then since $f$ is a continuous function in $x_{0}$, there exists $\delta>0$ such that if $\|y-x_{0}\|<\delta$, then $|y^{-\mu-1/2}f(y)-x_{0}^{-\mu-1/2}f(x_{0})|<\frac{\eps}{2x_{0}^{\mu+1/2}}$.

\begin{align*}
& f\convZ \phi_{m}(x_{0})-f(x_{0})
=
\int_{\Rnm}\int_{\Rnm}
y^{\mu+1/2}\phi_{m}(z)\:\DZ_{\mu}(x_{0},y,z)
[y^{-\mu-1/2}f(y)-x_{0}^{-\mu-1/2}f(x_{0})]\:dy\:dz\\
& =
\int_{\|z\|>\frac{\delta}{\sqrt{n}}}\int_{\Rnm}
y^{\mu+1/2}\phi_{m}(z)\:\DZ_{\mu}(x_{0},y,z)
[y^{-\mu-1/2}f(y)-x_{0}^{-\mu-1/2}f(x_{0})]\:dy\:dz\\
& +
\int_{\|z\|<\frac{\delta}{\sqrt{n}}}\int_{\Rnm}
y^{\mu+1/2}\phi_{m}(z)\:\DZ_{\mu}(x_{0},y,z)
[y^{-\mu-1/2}f(y)-x_{0}^{-\mu-1/2}f(x_{0})]\:dy\:dz.
\end{align*}

\noindent Calling
\begin{align*}
I_{1} &=
\int_{\|z\|>\frac{\delta}{\sqrt{n}}}\int_{\Rnm}
y^{\mu+1/2}\phi_{m}(z)\:\DZ_{\mu}(x_{0},y,z)
[y^{-\mu-1/2}f(y)-x_{0}^{-\mu-1/2}f(x_{0})]\:dy\:dz\\
I_{2} &=
\int_{\|z\|<\frac{\delta}{\sqrt{n}}}\int_{\Rnm}
y^{\mu+1/2}\phi_{m}(z)\:\DZ_{\mu}(x_{0},y,z)
[y^{-\mu-1/2}f(y)-x_{0}^{-\mu-1/2}f(x_{0})]\:dy\:dz,\\
\end{align*}

\noindent from where
\[
|f\convZ \phi_{m}(x_{0})-f(x_{0})| \leq |I_{1}|+|I_{2}|.
\]

\noindent Since $f\in L^{\infty}(r)$,

\[|y^{-\mu-1/2}f(y)-x_{0}^{-\mu-1/2}f(x_{0})|\leq 2\|f\|_{L^{\infty}(r)}.\]

\noindent Moreover, since $\lim\limits_{m\to\infty}\int_{\|z\|>\frac{\delta}{\sqrt{n}}} \phi_{m}(z)\:s(z)r(z)\:dz =0$, there exists $N_{0}\in\N$ such that
\[
\int_{\|z\|>\frac{\delta}{\sqrt{n}}} \phi_{m}(z)\:s(z)r(z)\:dz <
\frac{\eps}{4\|f\|_{L^{\infty}(r)} x_{0}^{\mu+1/2}},
\quad \forall\:m>N_{0}.
\]

\begin{align*}
|I_{1}| &\leq
\int_{\|z\|>\frac{\delta}{\sqrt{n}}}\int_{\Rnm}
y^{\mu+1/2}\phi_{m}(z)\:\DZ_{\mu}(x_{0},y,z)
|y^{-\mu-1/2}f(y)-x_{0}^{-\mu-1/2}f(x_{0})|\:dy\:dz\\
& =
2\|f\|_{L^{\infty}(r)}
\int_{\|z\|>\frac{\delta}{\sqrt{n}}}
\phi_{m}(z)
\left\{\int_{\Rnm}
y^{\mu+1/2}\:\DZ_{\mu}(x_{0},y,z)
\:dy\right\}dz\\
& =
2\|f\|_{L^{\infty}(r)}
\int_{\|z\|>\frac{\delta}{\sqrt{n}}}
\phi_{m}(z)C_{\mu}^{-1} x_{0}^{\mu+1/2} z^{\mu+1/2}
dz\\
& =
2\|f\|_{L^{\infty}(r)} x_{0}^{\mu+1/2}
\int_{\|z\|>\frac{\delta}{\sqrt{n}}}
\phi_{m}(z)\:s(z)r(z)\:dz\\
& <
2\|f\|_{L^{\infty}(r)} x_{0}^{\mu+1/2}
\frac{\eps}{4\|f\|_{L^{\infty}(r)} x_{0}^{\mu+1/2}}=\frac{\eps}{2}
\end{align*}

\noindent On the other hand, if $\|z\|<\frac{\delta}{\sqrt{n}}$, then $z_{i}\in(0,\delta/\sqrt{n})$ for all $i=1,\cdots,n$. Moreover, if we consider $\DZ_{\mu}(x_0,y,z)$ as a function depending on $y$ then

\begin{align*}
\supp\DZ_{\mu}(x_0,y,z)
&\subset
(|x_{1}^{0}-z_{1}|, x_{1}^{0}+z_{1})\times\cdots\times (|x_{n}^{0}-z_{n}|, x_{n}^{0}+z_{n}) \\
&\subset
(x_{1}^{0}-\delta/\sqrt{n}, x_{1}^{0}+\delta/\sqrt{n})\times\cdots\times (x_{n}^{0}-\delta/\sqrt{n}, x_{n}^{0}+\delta/\sqrt{n}).
\end{align*}
So, if $y\in\supp\DZ_{\mu}(x_0,y,z) $ then $|y_{i}-x_{i}^{0}|< \delta/\sqrt{n}$, for all $i=1,\cdots,n$. Thus
\[\|y-x_{0}\|_{\infty} = \max\limits_{1\leq i\leq n}\{|y_{i}-x_{i}^{0}|\}<\delta/\sqrt{n}.\]
\noindent Since the euclidean norm is equivalent to the uniform norm and $\|\cdot\|_{\infty}\leq\|\cdot\|\leq\sqrt{n} \|\cdot\|_{\infty}$, then we obtain that  $\|z\|<\frac{\delta}{\sqrt{n}}$ implies $\|y-x_{0}\|<\delta$, then

\begin{align*}
|I_{2}| &\leq
\int_{\|z\|<\frac{\delta}{\sqrt{n}}}\int_{\Rnm}
y^{\mu+1/2}\phi_{m}(z)\:\DZ_{\mu}(x_{0},y,z)
|y^{-\mu-1/2}f(y)-x_{0}^{-\mu-1/2}f(x_{0})|\:dy\:dz\\
&\leq
\int_{\|z\|<\frac{\delta}{\sqrt{n}}}\int_{\Rnm}
y^{\mu+1/2}\phi_{m}(z)\:\DZ_{\mu}(x_{0},y,z)
\frac{\eps}{2x_{0}^{\mu+1/2}}\:dy\:dz\\
& =
\frac{\eps}{2x_{0}^{\mu+1/2}}
\int_{\|z\|<\frac{\delta}{\sqrt{n}}}
\phi_{m}(z)
\left\{\int_{\Rnm}
y^{\mu+1/2}\:\DZ_{\mu}(x_{0},y,z)
\:dy\right\}dz\\
& =
\frac{\eps}{2x_{0}^{\mu+1/2}}
\int_{\|z\|<\frac{\delta}{\sqrt{n}}}
\phi_{m}(z)\:C_{\mu}^{-1}x_{0}^{\mu+1/2}z^{\mu+1/2}
dz\\
& =
\frac{\eps}{2}
\int_{\|z\|<\frac{\delta}{\sqrt{n}}}
\phi_{m}(z)\:s(z)r(z)\:dz
\leq
\frac{\eps}{2}
\int_{\Rnm}
\phi_{m}(z)\:s(z)r(z)\:dz = \frac{\eps}{2}.
\end{align*}

\noindent From where we have proved that given $\eps>0$, there exists $N_{0}\in\N$ such that $|f\convZ \phi_{m}(x_{0})-f(x_{0})|<\eps$, for all $n>N_{0}$. The uniform convergence is obtained analogously to the uniform continuity of $rf$.
\end{proof}

\bigskip
We are going to consider Bessel operators in $\Rnm$ given by \eqref{nOBZ} and \eqref{nOBH} which are related through
\begin{equation}\label{similaridad Op de Bessel}
\nBZ=x^{\mu+1/2}\nBH x^{-\mu-1/2},
\end{equation}
see remark \ref{proof similaridad Op de Bessel} for a proof.

\bigskip
Bessel operator \eqref{nOBZ}  and Hankel transform \eqref{nHZ} were studied in the distributional setting over the Zemanian spaces $\mathcal{H}_{\mu}$ and $\mathcal{H'}_{\mu}$ in \cite{Ze66} ($1$-dimensional case), \cite{Mo03} and \cite{MT07}  ($n$-dimensional case).

\bigskip

$S_{\mu}$ is a continuous  operator in $\mathcal{H}_{\mu}$ and selfadjoint, so the generalized Bessel operator  $S_{\mu}$ can be extended to $\mathcal{H'}_{\mu}$ by transposition
\begin{equation*}
(S_{\mu} f, \phi)=(f, S_{\mu}\phi),
\quad f\in\mathcal{H'}_{\mu}
\quad \phi\in\mathcal{H}_{\mu}.\\
\end{equation*}

\medskip

\noindent Analogously, generalized Hankel  transform $h_{\mu}f$ can be extended to $\mathcal{H'}_\mu$ by
\begin{equation*}
(h_{\mu}f, \phi)=(f, h_{\mu}\phi),
\quad f\in\mathcal{H'}_{\mu},
\quad \phi\in\mathcal{H}_\mu
\end{equation*}

\noindent for $\mu=(\mu_1,\ldots,\mu_n)$, $\mu_{i}>-\frac{1}{2}$, $i=1,\ldots,n$. Then $h_{\mu}$ is an automorfism over $\mathcal{H}_{\mu}$ and $\mathcal{H'}_{\mu}$.

\clearpage

There exist different proofs for the inversion theorem of the Hankel transform for the $1$-dimensional case. In this work we present a proof for the inversion theorem for the $n$-dimensional case, in the same way of the classic versions of the results known for the inversion of the Fourier transform in Lebesgue spaces.

\begin{theorem}\label{TeoInv-nHZ}
Let $f\in L^{1}(\Rnm,x^{\mu+1/2})$ and $\nHZ_{\mu}f\in L^{1}(\Rnm,x^{\mu+1/2})$ where $x^{\mu+1/2}$ is given by \eqref{x^b}. Then $f(x)$ may be redefined on a set of measure zero so that it is continuous on $\Rnm$ and
\begin{equation}\label{FormulaInv-nHZ}
f(x) = \nHZ_{\mu}(\nHZ_{\mu}f)(x),
\end{equation}
\noindent for almost every $x\in\Rnm$.
\end{theorem}
\begin{proof}
For the proof of this result we refer the reader to the Appendix. Details can be found in page \pageref{Proof-TeoInv-nHZ}.
\end{proof}
\medskip
\begin{remark}\label{TeoInv-nHZ en Hmu}
From Theorem \ref{TeoInv-nHZ} we deduce immediately the validity of equality \eqref{FormulaInv-nHZ} in $\mathcal{H}_{\mu}$ and $\mathcal{H'}_{\mu}$.
\end{remark}

\vspace{.2in}
\noindent For the proof of the next results we refer the reader to \cite{MT07}.

\begin{lemma}\label{lema3.8}
Let $\phi\in\mathcal{H}_{\mu}$, then
\begin{itemize}
\item[(i)] $\nHZ_{\mu}\nBZ\phi=-\|y\|^{2}\nHZ_{\mu}\phi$.
\item[(ii)]$\nBZ\nHZ_{\mu}\phi=\nHZ_{\mu}(-\|x\|^{2}\phi)$.
\end{itemize}
\end{lemma}

\begin{lemma}\label{lema3.10}
If $u\in\mathcal{H'}_{\mu}$, then
\begin{itemize}
\item[(i)] $\nHZ_{\mu}\nBZ u = -\|x\|^{2}\nHZ_{\mu} u$.
\item[(ii)]$\nBZ\nHZ_{\mu} u = \nHZ_{\mu}(-\|y\|^{2} u)$.
\end{itemize}
\end{lemma}

\begin{remark}\label{lema3.9}
According to Lemma 3.2 in \cite{Mo03} the functions $(\lambda+\|x\|^{2})$ for $\lambda\geq 0$ and  $(\lambda+\|x\|^{2})^{-1}$ for $\lambda> 0$ belong to the space of multipliers of $\mathcal{H}_{\mu}$ and $\mathcal{H'}_{\mu}$.
\end{remark}

 So, the next result holds.
\begin{lemma}\label{lema3.11}
The following equalities are valid in  $\mathcal{H}_{\mu}$ and in $\mathcal{H'}_{\mu}$ for $m\in\N$ and $\lambda\in\C$.
\begin{itemize}
\item[(i)] $(-\nBZ+\lambda)^{m}\nHZ_{\mu} = \nHZ_{\mu}(\lambda+\|y\|^{2})^{m}$.
\item[Si] $\lambda>0$,
\item[(ii)] $\nHZ_{\mu} (-\nBZ+\lambda)^{-m} = (\lambda+\|y\|^{2})^{-m}\nHZ_{\mu}$
\item[(iii)] $\nHZ_{\mu}(-\nBZ(-\nBZ+\lambda)^{-1})^{m}=
\|y\|^{2m}(\lambda+\|y\|^{2})^{-m}\nHZ_{\mu}$
\end{itemize}
\end{lemma}
\begin{proof}
We refer the reader to page \pageref{proof-lema3.11} in the Appendix for details.
\end{proof}

\section{Non-negativity and fractional powers of similar operators}
In this section we include a brief review of non-negative operators in Banach spaces and locally convex spaces.
\medskip
Let $X$ be a Banach space (real or complex). Let $A$ be a closed linear operator $A:D(A)\subset X\to X$ and $\rho(A)$ the resolvent set of $A$. We say that $A$ is non-negative if $(-\infty,0)\subset \rho(A)$ and
\[
\sup_{\lambda>0}
\{\|\lambda(\lambda+A)^{-1}\|\}<\infty.
\]

\vspace{.2in}
Now, let $X$ is a locally convex space with a Hausdorff topology generated by a directed family of seminorms $\{\|\:\:\|_{\alpha}\}_{\alpha\in\Lambda}$. A family of linear operators $\{A_{\lambda}\}_{\lambda\in\Gamma}$, $A_{\lambda}:D(A_{\lambda})\subset X\to X$, is equicontinuous if for each $\alpha\in\Lambda$ there are $\beta=\beta(\alpha)\in\Lambda$ and a constant $C=C_{\alpha}\geq 0$ such that for all $\lambda\in\Gamma$
\[
\|A_{\lambda}\phi\|_{\alpha}
\leq C\|\phi\|_{\beta},
\quad \phi\in X.
\]
Under the above conditions, we say that a closed linear operator $A:D(A)\subset X\to X$ is non-negative if $(-\infty,0)\subset\rho(A)$ and the family of operators
\[
\{\lambda(\lambda+A)^{-1}\}_{\lambda>0}
\]
is equicontinuous.

\vspace{.2in}
Now, we will briefly describe the theory of fractional powers of operators. According to  \cite[Proposition 3.1.3]{MS01}, we can define the Balakrishnan operator $J^{\alpha}$ in the following way.

\bigskip
Let $A$ be a non-negative operator in a Banach space or in a locally convex and sequentially complete space. Let
$\alpha\in\C_{+}$ and $n>\Re\alpha$, $n\in\N$. If $\phi\in D(A^{n})$ and $m\geq n$ is a positive integer, then
\begin{equation}\label{prop4.1-eq1}
J^{\alpha}\phi=
\frac{\Gamma(m)}{\Gamma(\alpha)\Gamma(m-\alpha)}
\int_{0}^{\infty}
\lambda^{\alpha-1}\big[A(\lambda+A)^{-1}\big]^{m}\phi\:d\lambda.
\end{equation}

\bigskip
If $A$ is bounded, $J_{A}^{\alpha}$ can be considered as the fractional power of $A$. In other cases we can consider the following representation for the fractional power stated in \cite[Theorem 5.2.1]{MS01}
\begin{theorem}
Let $A$ be a non-negative operator, $\alpha\in\C_{+}$, $\lambda\in\rho(-A)$ and $n\in\N$. Then
\begin{equation}\label{Theorem 5.2.1-MS01}
A^{\alpha}=(A+\lambda)^{n}J_{A}^{\alpha}(A+\lambda)^{-n}.
\end{equation}
(If $n>\Re\alpha$, the operator $\overline{J_{A}^{\alpha}}$ can be replaced by $J_{A}^{\alpha}$ in the preceding formula.)
\end{theorem}

\bigskip
Similar operators has been described in the introduction. Let $A$ and  $B$ similar operators and $T$ the isomorphism that verifies \eqref{OpSimilares} then
\[(zId-B)^{-1}=T(zId-A)^{-1}T^{-1},\]
\noindent for $z$ a complex number, from where we deduce immediately that  $A$ is non-negative operator if and only if so is $B$.

\bigskip
When two operators are similar, the fractional powers also meet this property. Thus we have the following result which holds in Banach spaces and in locally convex and sequentially complete spaces.

\begin{proposition}\label{prop-similares Banach}
Let  $A$ and $B$ be similar non-negative operators. If $\alpha\in\C_{+}$ then
\begin{equation}\label{prop4.1-1}
    J_{B}^{\alpha}=T J_{A}^{\alpha}T^{-1},
\end{equation}
and
\begin{equation}\label{prop4.1-2}
    B^{\alpha}=TA^{\alpha}T^{-1},
\end{equation}
where $T$ is the isometric isomorphism that verifies $B=TAT^{-1}$.
\end{proposition}

\section{Fractional powers of \texorpdfstring{$\nBZ$}{Smu}  in Lebesgue spaces}

Let $s$ and $r$ as in Section 2 and let $1\leq p<\infty$. We will denote by $S_{\mu,p}$ the part of $\nBZ$ in $L^{p}(\Rnm,sr^{p})$, that is to say, the operator $\nBZ$ with domain
\[
D(S_{\mu,p})=
\{f\in L^{p}(\Rnm,sr^{p}):
\nBZ f\in L^{p}(\Rnm,sr^{p})\}
\]
and given by $S_{\mu,p}f=\nBZ f$.

\medskip
Analogously, with $S_{\mu,\infty}$ we will denote the part of  $\nBZ$ in $L^{\infty}(\Rnm,r)$, $\Delta_{\mu,p}$ and  $\Delta_{\mu,\infty}$ the part of $\nBH$ in $L^{p}(\Rnm,s)$ and  $L^{\infty}(\Rnm)$ respectively.

\bigskip
\noindent Let $L_{r}$ the isometric isomorphism
\[
L_{r}: L^{p}(\Rnm,sr^{p})\to  L^{p}(\Rnm,s),
\quad\text{with}\quad
1\leq p<\infty
\]
(or $L_{r}: L^{\infty}(\Rnm,r)\to  L^{\infty}(\Rnm)$) given by  \[L_{r}(f)=rf.\]

Let
then
\[
S_{\mu,p}=
L_{r}^{-1}\circ
\Delta_{\mu,p}\circ L_{r}.
\]
\vspace{.2in}

Consequently it is enough to study the operator $\nBZ$ in the spaces $L^{p}(\Rnm,sr^{p})$ (or $L^{\infty}(\Rnm,r)$). In order to study the non-negativity of operators  $-S_{\mu,p}$ and $-S_{\mu,\infty}$ we consider the following function given by
\begin{equation}
\mathcal{N}_{\nu}(w)=
\int_{0}^{\infty}
e^{-t-\frac{w^{2}}{4t}}\:\frac{dt}{t^{\nu+1}}
\end{equation}
which is defined for all $\nu\in\R$ and $w\in\Rm$.

\bigskip
\noindent Let  $t\in\Rm$, if $\mu=(\mu_{1},\ldots,\mu_{n})$, then $t^{\mu+1}$ means
\[
t^{\mu+1}=
t^{\mu_{1}+1}\ldots t^{\mu_{n}+1}=
t^{\mu_{1}+\ldots+\mu_{n}+n},
\]
from where
\begin{equation}
\mathcal{N}_{\mu_{1}+\ldots+\mu_{n}+n-1}(\|x\|)
=\int_{0}^{\infty}
e^{-t-\frac{\|x\|^{2}}{4t}}\:\frac{dt}{t^{\mu_{1}+\ldots+\mu_{n}+n-1+1}}
=\int_{0}^{\infty}
e^{-t-\frac{\|x\|^{2}}{4t}}\:\frac{dt}{t^{\mu+1}}.
\end{equation}

\vspace{.2in}
\noindent Given $\lambda>0$, let us consider the function

\smallskip
\begin{equation}\label{nucleo}
N_{\lambda}(x)=
2^{-\mu-1}x^{\mu+1/2}\lambda^{\mu}\lambda^{n-1}
\mathcal{N}_{\mu_{1}+\ldots+\mu_{n}+n-1}(\|\sqrt{\lambda}x\|),
\quad x\in\Rnm.
\end{equation}

\medskip
\begin{lemma}\label{Lema4.3}
Given $\mu=(\mu_{1},\ldots,\mu_{n})$, $\mu_{i}>-\frac{1}{2}$ and $\lambda>0$ then
\begin{enumerate}
\item[(a)] $N_{\lambda}\in L^{1}(\Rnm,sr)$ and
\[\|N_{\lambda}\|_{L^{1}(\Rnm,sr)}=\frac{1}{\lambda}\]
\item[(b)]
\[\nHZ_{\mu}N_{\lambda}(y)=\frac{y^{\mu+1/2}}{\lambda+\|y\|^{2}}\]
\end{enumerate}
\end{lemma}

\begin{proof}[Proof (a)]\phantom{\qedhere}
\begin{align*}
\|N_{\lambda}\|_{L^{1}(\Rnm,sr)}
&=
\int_{\Rnm}|N_{\lambda}(x)|\:\frac{x^{\mu+1/2}}{C_{\mu}} dx\\
&=
\int_{\Rnm} 2^{-\mu-1} x^{\mu+1/2}\lambda^{\mu}\lambda^{n-1}
\mathcal{N}_{\mu_{1}+\ldots+\mu_{n}+n-1}(\|\sqrt{\lambda}x\|)
\:\frac{x^{\mu+1/2}}{C_{\mu}} dx\\
&=
2^{-\mu-1}\lambda^{\mu}\lambda^{n-1}\frac{1}{C_{\mu}}
\int_{\Rnm}\left\{\int_{0}^{\infty}
e^{-t-\frac{\lambda\|x\|^{2}}{4t}}\frac{dt}{t^{\mu+1}}
\right\} x^{2\mu+1} dx\\
&=
2^{-\mu-1}\lambda^{\mu}\lambda^{n-1}\frac{1}{C_{\mu}}
\int_{0}^{\infty}\left\{\int_{\Rnm}
e^{-\frac{\lambda\|x\|^{2}}{4t}}x^{2\mu+1} dx
\right\} e^{-t}\frac{dt}{t^{\mu+1}}\\
&=
2^{-\mu-1}\lambda^{\mu}\lambda^{n-1}\frac{1}{C_{\mu}}
\int_{0}^{\infty} \prod_{i=1}^{n}\left\{\int_{0}^{\infty}
e^{-\frac{\lambda x_{i}^{2}}{4t}} x_{i}^{2\mu_{i}+1} dx_{i}
\right\} e^{-t}\frac{dt}{t^{\mu+1}}\\
&=
2^{-\mu-1}\lambda^{\mu}\lambda^{n-1}\frac{1}{C_{\mu}}
\int_{0}^{\infty} \prod_{i=1}^{n}\left\{
2^{\mu_{i}}\Gamma(\mu_{i}+1) \left(\frac{2t}{\lambda}\right)^{\mu_{i}+1}
\right\} e^{-t}\frac{dt}{t^{\mu+1}}\\
&=
2^{-\mu-1}\lambda^{\mu}\lambda^{n-1}\frac{1}{C_{\mu}}
2^{\mu+1}\lambda^{-\mu}\lambda^{-n}C_{\mu} \int_{0}^{\infty}
t^{\mu+1} e^{-t}\frac{dt}{t^{\mu+1}}=\frac{1}{\lambda}
\end{align*}
where we have used the formula \eqref{ApA-eq2}.
\end{proof}

\begin{proof}[Proof (b)]
\begin{align*}
\nHZ_{\mu} & N_{\lambda}(y)=
\int_{\Rnm} N_{\lambda}(x)
\prod_{i=1}^{n}\{\sqrt{x_{i}y_{i}} J_{\mu_{i}}(x_{i}y_{i})\} dx\\
&=
\int_{\Rnm}
2^{-\mu-1}x^{\mu+1/2}\lambda^{\mu}\lambda^{n-1}
\left\{\int_{0}^{\infty}
e^{-t-\frac{\lambda\|x\|^{2}}{4t}}
\frac{dt}{t^{\mu+1}}
\right\}
\prod_{i=1}^{n}
\{\sqrt{x_{i}y_{i}}J_{\mu_{i}}(x_{i}y_{i})\} dx\\
&=
2^{-\mu-1}y^{1/2}\lambda^{\mu}\lambda^{n-1}
\int_{0}^{\infty} \left\{
\int_{\Rnm}
x^{\mu+1}e^{-\frac{\lambda\|x\|^{2}}{4t}}
\prod_{i=1}^{n}
\{J_{\mu_{i}}(x_{i}y_{i})\}\:dx
\right\} e^{-t} \frac{dt}{t^{\mu+1}}\\
&=
2^{-\mu-1}y^{1/2}\lambda^{\mu}\lambda^{n-1}
\int_{0}^{\infty} \prod_{i=1}^{n} \left\{
\int_{0}^{\infty}
x^{\mu_{i}+1} e^{-\frac{\lambda x_{i}^{2}}{4t}}
J_{\mu_{i}}(x_{i}y_{i})\:dx_{i}
\right\} e^{-t} \frac{dt}{t^{\mu+1}}\\
&=
2^{-\mu-1}y^{1/2}\lambda^{\mu}\lambda^{n-1}
\int_{0}^{\infty} \prod_{i=1}^{n} \left\{
\left(\frac{\lambda}{2t}\right)^{-\mu_{i}-1} y_{i}^{\mu_{i}}
e^{-\frac{ty_{i}^{2}}{\lambda}}
\right\} e^{-t} \frac{dt}{t^{\mu+1}}\\
&=
2^{-\mu-1}y^{1/2}\lambda^{\mu}\lambda^{n-1}
2^{\mu+1}\lambda^{-\mu-1}y^{\mu}
\int_{0}^{\infty}
t^{\mu+1}
e^{-\frac{t\|y\|^{2}}{\lambda}} e^{-t}
\frac{dt}{t^{\mu+1}}\\
&=
y^{\mu+1/2}\lambda^{\mu}\lambda^{n-1}
\lambda^{-\mu}\lambda^{-n}
\int_{0}^{\infty}
e^{-t(1+\frac{\|y\|^{2}}{\lambda})}
\:dt\\
&=
y^{\mu+1/2}\lambda^{-1} \frac{\lambda}{\lambda+\|y\|^{2}}
\int_{0}^{\infty} e^{-s}\:ds\\
&=
\frac{y^{\mu+1/2}}{\lambda+\|y\|^{2}}
\end{align*}
where we have used \eqref{Prop3.14-eq1}.
\end{proof}

\begin{lemma}
Let $1\leq p<\infty$. If $f\in L^{p}(\Rnm,sr^{p})$ or $f\in L^{\infty}(\Rnm,r)$ then the following equality holds on $\mathcal{H'}_{\mu}$
\begin{equation}
\nHZ_{\mu}(N_{\lambda}\convZ f) =
\frac{1}{\lambda+\|y\|^{2}} \nHZ_{\mu}f
\end{equation}
\end{lemma}

\begin{proof}
Suppose that  $f\in L^{p}(\Rnm,sr^{p})$ and  $\psi\in\mathcal{H}_{\mu}$, we claim that
\begin{equation}\label{Lema4.4-eq1}
 \int_{\Rnm}
 (N_{\lambda}\convZ f)(x)\psi(x)\;dx
 =
 \int_{\Rnm}
 f(z)(N_{\lambda}\convZ \psi)(z)\;dz
\end{equation}

\begin{equation}\label{Lema4.4-eq2}
 \int_{\Rnm}
 f(z)(N_{\lambda}\convZ \psi)(z)\;dz=
 \int_{\Rnm} f(z)
 \left\{
 \int_{\Rnm}\int_{\Rnm}
 N_{\lambda}(y)\psi(x)\;\DZ_{\mu}(x,y,z)\;dy\;dx
 \right\}dz.
\end{equation}

\noindent Let us see that $ \int_{\Rnm} |f(z)|
\left\{
\int_{\Rnm}\int_{\Rnm}
|N_{\lambda}(y)|\;|\psi(x)|\;\DZ_{\mu}(x,y,z)\;dy\;dx
\right\}dz$ is finite.

\noindent Let
\[G(z)=
\int_{\Rnm}\int_{\Rnm}
|N_{\lambda}(y)|\;|\psi(x)|\;\DZ_{\mu}(x,y,z)\;dy\;dx\]

\noindent and let $q$ such that $\frac{1}{p}+\frac{1}{q}=1$. The function $G$ is the convolution of  $|N_{\lambda}|$ and $|\psi|$. From Lemma \ref{lema3.3}, since $|N_{\lambda}|\in L^{1}(\Rnm,sr)$ and $|\psi|\in L^{q}(\Rnm,sr^{q})$ we have that $G\in L^{q}(\Rnm,sr^{q})$, then

\begin{align*}
\int_{\Rnm}
|f(z)|\;|G(z)|\;dz&=
\int_{\Rnm}
(r|f(z)|)\;(r^{-1}s^{-1}|G(z)|)\;s\;dz\\
&=
\int_{\Rnm}
(r|f(z)|)\;(C_{\mu}r|G(z)|)\;s\;dz\\
&=
C_{\mu}\int_{\Rnm}
|rf(z)|\;|rG(z)|\;s\;dz\\
&\leq
C_{\mu}\|rf\|_{L^{p}(\Rnm,s)}\|rG\|_{L^{q}(\Rnm,s)}\\
&=
C_{\mu}\|f\|_{L^{p}(\Rnm,sr^{p})}\|G\|_{L^{q}(\Rnm,sr^{q})}
\end{align*}
then it is possible to change the order of integration in \eqref{Lema4.4-eq2}.

\begin{align*}
\int_{\Rnm}f(z)\;(N_{\lambda}\convZ\psi)(z)\;dz
&=
\int_{\Rnm}
\left\{
\int_{\Rnm}\int_{\Rnm}
N_{\lambda}(y)\;f(z)\;\DZ_{\mu}(x,y,z)\;dy\;dz
\right\}\;\psi(x)\;dx\\
&=
\int_{\Rnm}
(N_{\lambda}\convZ f)(x)
\;\psi(x)\;dx\\
\end{align*}
So, we have proved \eqref{Lema4.4-eq2}.

\medskip
\noindent Now let  $f\in L^{\infty}(\Rnm,r)$. To see that \eqref{Lema4.4-eq1} holds, it will be enough to see that
\begin{align*}
\int_{\Rnm}&\left\{
\int_{\Rnm}\left\{
\int_{\Rnm}
|f(z)|\;|N_{\lambda}(y)|\;|\psi(x)|\;\DZ_{\mu}(x,y,z)\;dz
\right\}\;dy\right\}\;dx\\
&\leq
\|rf\|_{L^{\infty}(\Rnm)}
\int_{\Rnm}\left\{
\int_{\Rnm}
|N_{\lambda}(y)|\;|\psi(x)|\;
\left\{
\int_{\Rnm}
z^{\mu+1/2}\DZ_{\mu}(x,y,z)\;dz
\right\}\;dy\right\}\;dx\\
&= C_{\mu}\;\|f\|_{L^{\infty}(\Rnm,r)}
          \;\|N_{\lambda}\|_{L^{1}(\Rnm,sr)}
          \;\|\psi\|_{L^{1}(\Rnm,sr)}<\infty.
\end{align*}

\medskip
\noindent Let $\phi\in\mathcal{H}_{\mu}$ and $f\in L^{p}(\Rnm,sr)$ or $f\in L^{\infty}(\Rnm,r)$, from \eqref{Lema4.4-eq1} we have that
\begin{equation}\label{Lemma4.4-eq2}
(\nHZ_{\mu}(N_{\lambda}\convZ f),\phi) =
((N_{\lambda}\convZ f),\nHZ_{\mu}\phi) =
\int_{\Rnm} (N_{\lambda}\convZ f)(x)\;(\nHZ_{\mu}\phi)(x)\;dx=
\int_{\Rnm} f(z)\;(N_{\lambda}\convZ\nHZ_{\mu}\phi)(z)\;dz.
\end{equation}

\noindent From Lemma \ref{lema3.4},  Theorem \ref{TeoInv-nHZ} and item  (b) of Lemma \ref{Lema4.3} we obtain that
\[
\nHZ_{\mu}(N_{\lambda}\convZ \nHZ_{\mu}\phi)(y)=
r(\nHZ_{\mu}N_{\lambda})(\nHZ_{\mu}(\nHZ_{\mu}\phi))(y)=
y^{-\mu-1/2}\frac{y^{\mu+1/2}}{\lambda+\|y\|^{2}}\phi(y)=
\frac{\phi(y)}{\lambda+\|y\|^{2}}.
\]

\noindent Then
\begin{equation}\label{Lemma4.4-eq3}
N_{\lambda}\convZ \nHZ_{\mu}\phi=
\nHZ_{\mu}\left(\frac{\phi}{\lambda+\|y\|^{2}}\right).
\end{equation}

\noindent Finally, from \eqref{Lemma4.4-eq2} and \eqref{Lemma4.4-eq3} we obtain that for $\phi\in\mathcal{H}_{\mu}$ that

\begin{align*}
(\nHZ_{\mu}(N_{\lambda}\convZ f),\phi)
&=
\int_{\Rnm}
f(x)(N_{\lambda}\convZ\nHZ_{\mu}\phi)(x)\;dx\\
&=
\int_{\Rnm}
f(x)\;\nHZ_{\mu}\left(\frac{\phi}{\lambda+\|y\|^{2}}\right)(x)\;dx\\
&=
\int_{\Rnm}
\frac{1}{\lambda+\|x\|^{2}}
\nHZ_{\mu}f(x)\;\phi(x)\;dx\\
&=
\left(\frac{\nHZ_{\mu}f}{\lambda+\|x\|^{2}},\phi\right)
\end{align*}
\end{proof}

\begin{theorem}\label{Teorema4.7}
Given $\mu=(\mu_{1},\ldots,\mu_{n})$, $\mu_{i}>-\frac{1}{2}$, then $S_{\mu,p}$ and $S_{\mu,\infty}$ are closed and non-negative operators.
\end{theorem}

\begin{proof}  Since convergence in  $L^{\infty}(\Rnm,r)$ and $L^{p}(\Rnm,sr^p)$ implies convergence in  $\mathcal{D'}(\Rnm)$, then $S_{\mu,\infty}$ and $S_{\mu,p}$ are closed.

\medskip
Now let $\lambda>0$ and $f\in D(S_{\mu,\infty})$ such that  $(\lambda-S_{\mu,\infty})f=0$. So,
\[\nHZ_{\mu}(\lambda-S_{\mu,\infty})f=0\]
in $\mathcal{H'}_{\mu}$. By Lemma \ref{lema3.10} we obtain that
\[(\lambda+\|y\|^{2})\nHZ_{\mu}f=0\]
in  $\mathcal{H'}_{\mu}$ and hence by Lemma \ref{lema3.9}
\[\nHZ_{\mu}f=(\lambda+\|y\|^{2})^{-1}(\lambda+\|y\|^{2})\nHZ_{\mu}f=0.\]
Then, $f=0$ as element of $\mathcal{H'}_{\mu}$ and we conclude that  $f=0$ a.e. in $x\in\Rnm$ and  $\lambda-S_{\mu,\infty}$ is injective.

\medskip
Let $f\in L^{\infty}(\Rnm,r)$ and $g=N_{\lambda}\convZ f$. Then, by Lemma \ref{lema3.3} $g\in L^{\infty}(\Rnm,r)$ and
\[
\nHZ((\lambda-S_{\mu,\infty})g)=
(\lambda+\|y\|^{2})\nHZ_{\mu}g=
(\lambda+\|y\|^{2})\nHZ_{\mu}(N_{\lambda}\convZ f)=
\nHZ_{\mu}f.
\]

\noindent By injectivity of Hankel transform in $\mathcal{H'}_{\mu}$ we obtain that
\[(\lambda-S_{\mu,\infty})g=f,\]
so, $\lambda-S_{\mu,\infty}$ is onto. Also
\begin{align*}
\|(\lambda-S_{\mu,\infty})^{-1}f\|_{L^{\infty}(\Rnm,r)}
&=
\|g\|_{L^{\infty}(\Rnm,r)}=
\|N_{\lambda}\convZ f\|_{L^{\infty}(\Rnm,r)}\\
&\leq
\|N_{\lambda}\|_{L^{1}(\Rnm,sr)}\|f\|_{L^{\infty}(\Rnm,r)}\\
&=
\frac{1}{\lambda}\|f\|_{L^{\infty}(\Rnm,r)},
\end{align*}
hence
\[
\|\lambda(\lambda-S_{\mu,\infty})^{-1}f\|_{L^{\infty}(\Rnm,r)}
\leq \|f\|_{L^{\infty}(\Rnm,r)}\]
and $-S_{\mu,\infty}$ is non-negative.

\medskip
\noindent The proof of the non-negativity of $S_{\mu,p}$ is similar.

\end{proof}

\vspace{.2in}
Since we have proved that both  $-S_{\mu,p}$ and $-S_{\mu,\infty}$ are non-negative we can consider the fractional powers of them. If $\alpha\in\C$, $\Re(\alpha)>0$ and  $n>\Re(\alpha)$ then the fractional power of $-S_{\mu,\infty}$ can be represented from \eqref{Theorem 5.2.1-MS01} by:
\[(-S_{\mu,\infty})^{\alpha}=
(-S_{\mu,\infty}+1)^{n}\mathcal{J}_{\infty}^{\alpha}(-S_{\mu,\infty}+1)^{-n},
\]
where with $\mathcal{J}_{\infty}^{\alpha}$ we denote the Balakrishnan operator associated to $-S_{\mu,\infty}$ given by:
\[
\mathcal{J}_{\infty}^{\alpha}\phi=
\frac{\Gamma(n)}{\Gamma(\alpha)\Gamma(n-\alpha)}
\int_{0}^{\infty} \lambda^{\alpha-1}
[-S_{\mu,\infty}(\lambda-S_{\mu,\infty})^{-1}]^{n}\phi\;d\lambda,
\]
for $\alpha\in\C$, $0<\Re(\alpha)<n$ and $\phi\in D[(-S_{\mu,\infty})^{n}]$.

\medskip
Analogously for the representation of  fractional powers of $-S_{\mu,p}$.

\section{Non-negativity of Bessel operator \texorpdfstring{$\nBZ$}{Smu} in  the space \texorpdfstring{$\mathcal{B}$}{B}}

\begin{remark}
The operator $-\nBZ$ is not non-negative in $\mathcal{H}_{\mu}$.
\end{remark}
\noindent If $-\nBZ$ were non-negative in $\mathcal{H}_{\mu}$, since  $-\nBZ$ is continuous in $\mathcal{H}_{\mu}$, given $\alpha\in\C$, $0<\alpha<1$ and according to \eqref{prop4.1-eq1} and  \eqref{complementosEuler}, we have that fractional power $(-\nBZ)^{\alpha}$ would be given by
\begin{equation}\label{eq1-obs4.7}
(-\nBZ)^{\alpha}\phi=
\frac{\sin\alpha\pi}{\pi}
\int_{0}^{\infty}
\lambda^{\alpha-1}
(-\nBZ)(\lambda-\nBZ)^{-1}\phi\:d\lambda
\end{equation}
and  $D[(-\nBZ)^{\alpha}]=D(-\nBZ)=\mathcal{H}_{\mu}$. Applying the Hankel transform in  \eqref{eq1-obs4.7} we obtain

\begin{align*}
\nHZ_{\mu}(-\nBZ)^{\alpha}\phi
&=
\frac{\sin\alpha\pi}{\pi}
\int_{0}^{\infty}
\lambda^{\alpha-1}
\nHZ_{\mu}[(-\nBZ)(\lambda-\nBZ)^{-1}\phi]\:d\lambda\\
&=
\frac{\sin\alpha\pi}{\pi}
\int_{0}^{\infty}
\lambda^{\alpha-1} \|y\|^{2}(\lambda+\|y\|^{2})^{-1}
\nHZ_{\mu}\phi(y)\:d\lambda\\
&=
(\|y\|^{2})^{\alpha}\nHZ_{\mu}\phi(y),
\end{align*}
where we have interchanged the Bochner integral with the Hankel transform, and the we have applied item  $(iii)$ of  Lemma \ref{lema3.11} and \cite[Remark 3.1.1]{MS01}. This would imply that $(\|y\|^{2})^{\alpha}\nHZ_{\mu}\phi(y)\in\mathcal{H}_{\mu}$ which is not true in general.

\vspace{.5in}
Now we consider the Banach space  $Y=L^{1}(\Rnm,sr)\cap L^{\infty}(\Rnm,r)$, with norm
\[
\|f\|_{Y}=
\max\bigl\{
\|f\|_{L^{1}(\Rnm,sr)},
\|f\|_{L^{\infty}(\Rnm,r)}
\bigr\},
\]
and the part of the Bessel operator in $Y$, $(\nBZ)_{Y}$, with domain given by \[
D[(\nBZ)_{Y}]=
\{f\in Y:\: \nBZ f\in Y\}.
\]

\noindent From Theorem \ref{Teorema4.7} we have that $-(\nBZ)_{Y}$ is closed and non-negative.

\clearpage
\begin{proposition}\label{Prop4.7}
If $k>\frac{n}{2}$ then  $D[((\nBZ)_{Y})^{k+1}]\subset C_{0}(\Rnm)$.
\end{proposition}

\begin{proof}
\[
D[((\nBZ)_{Y})^{k+1}]=
\{
\phi\in D[((\nBZ)_{Y})^{k}]:
\quad
((\nBZ)_{Y})^{k}\phi\in D[(\nBZ)_{Y}]
\}
\]

\noindent Let $f\in D[((\nBZ)_{Y})^{k+1}]$, then $f$ and $((\nBZ)_{Y})^{k}f$ are in  $D[(\nBZ)_{Y}]$.

\bigskip
\noindent From Lemma \ref{lema3.1} and  \eqref{prop7.5 [Mo18]-eq1} we have that
\[
L^{1}(\Rnm,sr)\cap L^{\infty}(\Rnm,r)
\subset
L^{1}(\Rnm)\cap L^{2}(\Rnm).
\]
Then $f$ and $((\nBZ)_{Y})^{k}f$ are in $L^{1}(\Rnm)$.  From Remark \ref{remark7.4 [Mo18]} we obtain that
$\nHZ_{\mu}f$ and $\nHZ_{\mu}((\nBZ)_{Y})^{k}f$ are in $L^{\infty}(\Rnm)$,
that is to say that there exist $M>0$ such that
\[
|(1+\|y\|^{2k})\:\nHZ_{\mu}f|
\leq M.
\]

\noindent Since for $k>\frac{n}{2}$,  $(1+\|y\|^{2k})^{-1}$ is integrable in $\Rnm$, then $\nHZ_{\mu}f\in L^{1}(\Rnm)$. Then, we have proved that $f\in D[((\nBZ)_{Y})^{k+1}]$, $f$ and $\nHZ_{\mu}f\in  L^{1}(\Rnm)\cap L^{2}(\Rnm)$.

\medskip
\noindent From Remark \ref{L^2 subset H'mu} we have that $L^{1}(\Rnm)\cap L^{2}(\Rnm)\subset \mathcal{H'}_{\mu}$ and from Remark \ref{TeoInv-nHZ en Hmu}
we have
\[
\nHZ_{\mu}(\nHZ_{\mu}f)(x)=f(x),
\quad\text{a.e}\quad x\in\Rnm,
\]
considering $f$ as a regular distribution in $\mathcal{H'}_{\mu}$. Since $\nHZ_{\mu}f\in L^{1}(\Rnm)$, then by Proposition \ref{prop7.5 [Mo18]} we have that   $f=g\quad\text{a.e.}$
in $\Rnm$ with $g\in C_{0}(\Rnm)$.
\end{proof}

\vspace{.2in}
We now consider the following space:
\begin{equation}\label{eqB}
\mathcal{B}=
\{f\in Y: \:(\nBZ)^{k}f\in Y
\quad\text{for}\quad
k=0,1,2,\ldots\}
=\bigcap\limits_{k=0}^{\infty}
D[((\nBZ)_{Y})^{k}],
\end{equation}
with seminorms
\[
\rho_{m}(f)=\max_{0\leq k\leq m}
\bigl\{
\|(\nBZ)^{k}f\|_{Y},
\quad m=0,1,2,\ldots
\bigr\}.
\]

\begin{remark}\label{Remark4.8 bis}
From proposition \ref{Prop4.7} is evident that  $\mathcal{B}\subset C_{0}(\Rnm)$. Moreover, from Lemma \ref{lema3.1} we obtain that  $\mathcal{B}\subset L^{p}(\Rnm,sr^{p})$ for all $1\leq p<\infty$, and considering that  $S_{\mu}$ is a continuous operator from  $\mathcal{H}_{\mu}$ in itself then $\mathcal{H}_{\mu}\subset\mathcal{B}$ and the topology of $\mathcal{H}_{\mu}$ induced by $\mathcal{B}$ is weaker than the usual topology generated by the seminorms given by  \eqref{Topologia H_mu}. In fact, from \eqref{lema3.1-eq5} and \eqref{lema3.1-eq6} we have that

\begin{equation}
\|\phi\|_{Y}\leq
C\{\gamma_{0,0}^{\mu}(\phi)+\gamma_{m,0}^{\mu}(\phi)\},
\quad \phi\in\mathcal{H}_{\mu}.
\end{equation}
for $m>2\mu_i+2$, $i=1,\ldots,n$ and by the continuity of  $S_{\mu}$ in $\mathcal{H}_{\mu}$ we deduce that given a seminorm $\rho_{m}$, there exists a finite set of seminorms  $\{\gamma_{m_i,k_i}^{\mu}\}_{i=1}^{r}$ and constants $c_1,\ldots,c_r$ such that

\[
\rho_{m}(\phi)\leq
\sum_{i=1}^{r} c_{i}\:
\gamma_{m_i,k_i}^{\mu}(\phi),
\quad \phi\in\mathcal{H}_{\mu}.
\]
\noindent From the density of $\mathcal{D}(\Rnm)$ in $\mathcal{B}$ we deduce the density of $\mathcal{H}_{\mu}$ in $\mathcal{B}$.
\end{remark}

\bigskip
We denote with $(\nBZ)_{\mathcal{B}}$ the part of Bessel operator $\nBZ$ in $\mathcal{B}$, so the domain of the operator $(\nBZ)_{\mathcal{B}}$ is $\mathcal{B}$ and the following result holds.

\begin{theorem}\label{Teorema4.8}
$\mathcal{B}$ is a Fréchet space and $-(\nBZ)_{\mathcal{B}}$  is a continous and non-negative operator on $\mathcal{B}$.
\end{theorem}

\begin{proof}
Let $\{\phi_{k}\}$ a Cauchy sequence in $\mathcal{B}$, then the convergence of $\{\phi_{k}\}$ follows  considering the seminorm $\rho_{0}$ and the completeness of $L^{1}(\Rnm,sr)$ and $L^{\infty}(\Rnm,r)$.

Since $\rho_{m}(S_{\mu}\phi)=\rho_{m+1}(\phi)$ then  $(\nBZ)_{\mathcal{B}}$ is continuous. 
The non-negativity follows from Proposition  1.4.2 in \cite{MS01}.

\end{proof}

\section{Non-negativity of Bessel operator \texorpdfstring{ $\nBZ$}{Smu} in the distributional space \texorpdfstring{$\mathcal{B'}$}{B}}

We will study the non-negativity of Bessel operator in the topological dual space of $\mathcal{B}$ with the strong topology, that is to say, the space   $\mathcal{B'}$ endowed with the topology generated by the family of seminorms  $\{|\cdot|_{B}\}$, where the sets $B$ are bounded sets in $\mathcal{B}$, and the seminorms are given by

\[
|T|_{B}=
\sup\limits_{\phi\in B} |(T,\phi)|,
\quad T\in\mathcal{B'}.
\]

\begin{proposition}\label{B' secuencialmente completo}
$\mathcal{B'}$ is sequentially complete.
\end{proposition}

\begin{proof}
Let $\{T_{m}\}\subset\mathcal{B'}$ a Cauchy sequence, then for all bounded set $B\subset\mathcal{B}$,
\[
|T_{k}-T_{m}|_{B}\to 0,
\quad\text{for}\quad
k,m\to\infty,
\]
i.e, for all  $\eps>0$ there exists $N$ such that for all $k,m\geq N$ then
\[
|T_{k}-T_{m}|_{B}<\eps.
\]
So, in particular, since the unit sets $\{\phi\}\subset\mathcal{B}$ are bounded,
\[
|T_{k}-T_{m}|_{\{\phi\}}
=
|(T_{k}-T_{m},\phi)|
<\eps,
\quad\forall\:k,m\geq N,
\]
then $\{(T_{m},\phi)\}$ is a Cauchy sequence in $\C$, with which is convergent and there exists $T:\mathcal{B}\to\C$ such that
\[
(T,\phi)=
\lim\limits_{m\to\infty}(T_{m},\phi).
\]

\noindent Since $\mathcal{B}$ is barrelled, for being a Fréchet space, and from a generalization of the Banach–Steinhaus theorem  (see \cite[Theorem 4.7, pp.86]{Schaefer}), it has to $T\in\mathcal{B'}$.
\end{proof}

\bigskip
\begin{remark}\label{Remark4.9} $\quad L^{p}(\Rnm,sr^{p})$ and $L^{\infty}(\Rnm,r)$ are included in $\mathcal{B'},\quad(1\leq p<\infty)$.

\medskip
\noindent Let $f\in L^{p}(\Rnm,sr^{p})$, $\phi\in\mathcal{B}$ and $q$ such that $\frac{1}{p} + \frac{1}{q} = 1$, then
\begin{align}
\left|\int_{\Rnm}f(x)\phi(x)\:dx\right|
&=
\left|\int_{\Rnm}
f(x)\phi(x)s^{-1}(x)r^{-p}(x)s(x)r^{p}(x)\:dx\right|\nonumber\\
&\leq
\|f\|_{L^{p}(\Rnm,sr^{p})}
\|\phi\:s^{-1}r^{-p}\|_{L^{q}(\Rnm,sr^{p})}\label{Remark4.9-eq1}
\end{align}

\noindent and
\begin{align}
\|\phi\:s^{-1}r^{-p}\|_{L^{q}(\Rnm,sr^{p})}
&=
\left\{
\int_{\Rnm}
|\phi\:s^{-1}r^{-p}|^{q} sr^{p}
\right\}^{\frac{1}{q}}
=
\left\{
\int_{\Rnm}
|\phi|^{q}\:(C_{\mu}r^{2}r^{-p})^{q} sr^{p}
\right\}^{\frac{1}{q}}\nonumber\\
&=
C_{\mu}
\left\{
\int_{\Rnm}
|\phi|^{q}\:r^{2q-pq+p} s
\right\}^{\frac{1}{q}}
=
C_{\mu}
\left\{
\int_{\Rnm}
|\phi|^{q}\: sr^{q}
\right\}^{\frac{1}{q}}\label{Remark4.9-eq2}.
\end{align}
\end{remark}

\noindent Furthermore, from \eqref{lema3.1-eq4}
\[
\|\phi\|_{L^{q}(\Rnm,sr^{q})}
\leq
\left\{\|\phi\|_{L^{\infty}(\Rnm,r)}\right\}^{\frac{q-1}{q}}
\left\{\|\phi\|_{L^{1}(\Rnm,sr)}\right\}^{\frac{1}{q}},
\]

\noindent from where
\begin{equation}\label{Remark4.9-eq3}
\|\phi\|_{L^{q}(\Rnm,sr^{q})}
\leq \rho_{0}(\phi).
\end{equation}

\noindent Then, from \eqref{Remark4.9-eq1}, \eqref{Remark4.9-eq2} and \eqref{Remark4.9-eq3} we obtain that $f\in\mathcal{B'}$.

\medskip
Now, let $B$ a bounded set in $\mathcal{B}$, then
\[
|f|_{B}
=\sup_{\phi\in B} \left|\int_{\Rnm}f\phi\right|
\leq C_{\mu}\|f\|_{L^{p}(\Rnm,sr^{p})}
\sup_{\phi\in B} \|\phi\|_{L^{q}(\Rnm,sr^{q})}
\leq C_{\mu}\|f\|_{L^{p}(\Rnm,sr^{p})}
\sup_{\phi\in B}\:\rho_{0}(\phi)
\]

Thus, the topology in $L^{p}(\Rnm,sr^{p})$ induced by $\mathcal{B'}$ with the strong topology is weaker than the usual topology.

\begin{remark}
Since $\mathcal{H}_{\mu}$ is dense in $\mathcal{B}$ and since the topology of $\mathcal{H}_{\mu}$ induced by $\mathcal{B}$ is weaker than the generated by the seminorms given by \eqref{Topologia H_mu} then  $\mathcal{B'}\subset\mathcal{H'}_{\mu}$. Moreover, from the continuity of the Bessel operator in $\mathcal{B}$, we can consider $\nBZ$ in $\mathcal{B'}$ as the adjoint operators of $\nBZ$ in $\mathcal{B}$, that is to say
\[
(\nBZ T,\phi)=(T,\nBZ\phi),
\quad T\in\mathcal{B'},
\phi\in\mathcal{B},
\]
and we denote with $(\nBZ)_{\mathcal{B'}}$ the part of Bessel operator in $\mathcal{B'}$.
\end{remark}

\begin{theorem}\label{Teorema4.9}
The operator $-(\nBZ)_{\mathcal{B'}}$ is continuous and non-negative considering the strong topology in $\mathcal{B'}$.
\end{theorem}

\begin{proof}

Given a bounded set  $B\subset\mathcal{B}$ and  $T\in\mathcal{B'}$, then
\[
|(\nBZ)_{\mathcal{B'}}T|_{B}=
\sup\limits_{\phi\in B}
|((\nBZ)_{\mathcal{B'}}T,\phi)|=
\sup\limits_{\phi\in B}
|(T,(\nBZ)_{\mathcal{B}}\phi)|=
|T|_{E}
\]
where the set  $E=\{(\nBZ)_{\mathcal{B}}\phi:\:\phi\in B\}$ is also bounded. Then it follows that $(\nBZ)_{\mathcal{B'}}$ is continuous.

\bigskip
\noindent Let now $\lambda>0$ and $T\in\mathcal{B'}$. It is not difficult to see that the linear map $G:\psi\to (T,(\lambda-(\nBZ)_{\mathcal{B}})^{-1}\psi)$ is continuous and $(\lambda-(\nBZ)_{\mathcal{B'}})G=T$.
\noindent Therefore $(\lambda-(\nBZ)_{\mathcal{B'}})$ is surjective.

\bigskip
\noindent To prove the injectivity, let $T\in\mathcal{B'}$ be such that $(\lambda-(\nBZ)_{\mathcal{B'}})T=0$. Then, for all  $\phi\in\mathcal{B}$,
\[
((\lambda-(\nBZ)_{\mathcal{B'}})T,\phi)=
(T,(\lambda-(\nBZ)_{\mathcal{B}})\phi)=0,
\]
and thus $T=0$ as $R(\lambda-(\nBZ)_{\mathcal{B}})={\mathcal{B}}$, due to the fact that $-(\nBZ)_{\mathcal{B}}$ is a non-negative operator.

\bigskip
\noindent To see that  $(\lambda-(\nBZ)_{\mathcal{B'}})^{-1}$ is continuous let $T\in\mathcal{B'}$, $B\subset\mathcal{B}$ a bounded set and let us consider the set $F=\{(\lambda-(\nBZ)_{\mathcal{B}})^{-1}\phi:\:\phi\in B\}$, then
\[
|(\lambda-(\nBZ)_{\mathcal{B'}})^{-1}T|_{B}
=
|G|_{B}
=
\sup\limits_{\psi\in B}
|(G,\psi)|
=
\sup\limits_{\psi\in B}
|(T,(\lambda-(\nBZ)_{\mathcal{B}})^{-1}\psi)|_{B}
=
|T|_{F}.
\]

\noindent For every bounded set $B\subset\mathcal{B}$ and $T\in\mathcal{B'}$, since $-(\nBZ)_{\mathcal{B}}$ is non-negative, the set $D=\{\eta(\eta-(\nBZ)_{\mathcal{B}})^{-1}\phi:\:\phi\in B,\:\eta>0\}$ is also bounded and thus, for $\lambda>0$,
\begin{align*}
|\lambda(\lambda-(\nBZ)_{\mathcal{B'}})^{-1}T|_{B}
&=
\sup\limits_{\phi\in B}
|(\lambda(\lambda-(\nBZ)_{\mathcal{B'}})^{-1}T,\phi)|\\
&=
\sup\limits_{\phi\in B}
|(T,\lambda(\lambda-(\nBZ)_{\mathcal{B}})^{-1}\phi)|\\
&\leq
|T|_{D}.
\end{align*}
We now conclude that the operator $-(\nBZ)_{\mathcal{B'}}$ is non-negative.
\end{proof}

\begin{remark}
The operator $(\nBZ)_{\mathcal{B'}}$ is not injective because the function $x^{\mu+\frac{1}{2}}$ is solution of $\nBZ u=0$ and belongs to $\mathcal{B'}$, in fact
\[
|(x^{\mu+\frac{1}{2}},\phi)|
\leq
C_{\mu}\|\phi\|_{L^{1}(\Rnm,sr)}
\leq
C_{\mu}\rho_{0}(\phi),
\quad \phi\in\mathcal{B}.
\]
\end{remark}

\vspace{.2in}

According to the representation of fractional powers of operators in locally convex spaces given in  \cite{MS01}, it has to  $(-(\nBZ)_{\mathcal{B'}})^{\alpha}$ is given by
\[
(-(\nBZ)_{\mathcal{B'}})^{\alpha}T=
\frac{\Gamma(n)}{\Gamma(\alpha)\Gamma(n-\alpha)}
\int_{0}^{\infty}
\lambda^{\alpha-1}
[-(\nBZ)_{\mathcal{B'}}
(\lambda-(\nBZ)_{\mathcal{B'}})^{-1}
]^{n}\:T\:d\lambda.
\]
for $\Re\alpha>0$, $n>\Re\alpha$, $T\in\mathcal{B'}$.

\vspace{.3in}
From the general theory of fractional powers in sequentially complete locally convex spaces  (see \cite[pp.134]{MS01}), we deduce some properties of powers such as multiplicativity and

\begin{itemize}
\item[(1)] If $\Re\alpha>0$ then
\begin{equation}\label{Obs-1}
\left((-(\nBZ)_{\mathcal{B}})^{\alpha}\right)^{*}=
\left((-(\nBZ)_{\mathcal{B}})^{*}\right)^{\alpha}
\end{equation}

\noindent Since $(-(\nBZ)_{\mathcal{B}})^{*}=-(\nBZ)_{\mathcal{B'}}$ then from \eqref{Obs-1} we obtain the following duality formula
\[
( (-(\nBZ)_{\mathcal{B'}})^{\alpha}T , \phi)
=
( T, (-(\nBZ)_{\mathcal{B}})^{\alpha} \phi),
\quad \phi\in\mathcal{B},
        T\in\mathcal{B'}.
\]

\item[(2)] Since the usual topology in $L^{p}(\Rnm,sr^{p})$ is stronger than the topology induced by  $\mathcal{B'}$ we can deduce that
\[
((-(\nBZ)_{\mathcal{B'}})^{\alpha})_{L^{p}(\Rnm,sr^{p})}
=
(-(S_{\mu,p}))^{\alpha},
\]
for $\Re\alpha>0$, (see \cite[Theorem 12.1.6, pp.284]{MS01}).

\vspace{.2in}
This last property expresses a very desirable property in the theory of powers since it tells us that the restriction  of the distributional power of  $-\nBZ$ to $L^{p}(\Rnm,sr^{p})$  coincides with the power of $-\nBZ$ in $L^{p}(\Rnm,sr^{p})$.
\end{itemize}

\section{Distributional Liouville theorem for \texorpdfstring{$(-(S_{\mu}))^{\alpha}$}{Smualpha}.}

In this section we include the proof of Theorem \ref{Teorema5.1}. Before that, we will show the following Lemma.

\begin{lemma}\label{lemmamultpot}
Let $\psi\in\mathcal{H}_{\mu}$ such that $\supp\psi\subset \Rnm\cap\{x:\|x\|\geq a\}$ with $a>0$ and  $\alpha\in \C$ with $\Re\alpha>0$. Then  $\|x\|^{-2\alpha}\psi(x)\in\mathcal{H}_{\mu}$.
\end{lemma}

\begin{proof} It is evident that $\|x\|^{-2\alpha}\psi(x)\in C^{\infty}(\Rnm)$. We are going to see that
\[
\sup_{x\in\Rnm}
\Bigl|x^{m}T^{k}
\{x^{-\mu-\frac{1}{2}}\|x\|^{-2\alpha}\psi(x)\}
\Bigr|<\infty,
\]
with $k,m\in\N_{0}^{n}$.
Since $\supp\psi\subset \Rnm\cap\{x:\|x\|\geq a\}$ with $a>0$, then
\begin{align*}
&\sup_{x\in\Rnm}
\Bigl|x^{m}T^{k}
\{x^{-\mu-\frac{1}{2}}\|x\|^{-2\alpha}\psi(x)\}\Bigr|
=
\sup_{x\in\Rnm: \|x\|\geq a}
\Bigl|x^{m}T^{k}
\{x^{-\mu-\frac{1}{2}}\|x\|^{-2\alpha}\psi(x)\}\Bigr|\\
&\quad\leq
\sup_{a\leq\|x\|\leq 1}\Bigl|x^{m}T^{k}
\{x^{-\mu-\frac{1}{2}}\|x\|^{-2\alpha}\psi(x)\}\Bigr|
+
\sup_{\|x\|\geq 1}\Bigl|x^{m}T^{k}
\{x^{-\mu-\frac{1}{2}}\|x\|^{-2\alpha}\psi(x)\}\Bigr|
\end{align*}
The first term in the last inequality is bounded because it is a continuous function  over a compact set. On the other hand, since equality \eqref{Leibniz} holds then,

\[
\sup_{\|x\|\geq 1}
\Bigl|x^{m}T^{k}
\{x^{-\mu-\frac{1}{2}}\|x\|^{-2\alpha}\psi(x)\}\Bigr|\leq
\]
\[
\sup_{\|x\|\geq 1}
\Bigl|x^{m}\sum_{j=0}^{k}
\binom{k}{j}T^{k-j}
\{x^{-\mu-\frac{1}{2}}\psi(x)\}\cdot
T^{j}\|x\|^{-2\alpha}\Bigr|\leq
\]
\[
\sum_{j=0}^{k} \binom{k}{j} C(j,\alpha)\:\:\gamma_{m,k-j}^{\mu}(\psi),
\]
where $C(j,\alpha)$ are constants depending on $\alpha$ and $j$ such that $\sup\limits_{\|x\|\leq 1} |T^{j}\|x\|^{-2\alpha}|\leq C(j,\alpha)$.

\end{proof}

\begin{proof}[Proof of Theorem \ref{Teorema5.1}]
Let $u\in \mathcal{B'}$ such that $(-(S_{\mu})_{\mathcal{B'}})^{\alpha}u=0$. Then for all $\phi\in \mathcal{B}$
\begin{equation}\label{eqequal0}
((-(S_{\mu})_{\mathcal{B'}})^{\alpha}u,\phi)=
(u,(-(S_{\mu})_{\mathcal{B}})^{\alpha}\phi)=0.
\end{equation}

Since $S_{\mu}$ is a continuous operator in  $\mathcal{B}$ (see Theorem  \ref{Teorema4.8}), then  $(-(S_{\mu})_{\mathcal{B}})^{\alpha}\phi$ is given by the Balakrishnan operator as:

\begin{equation}\label{eqpot1}
(-(S_{\mu})_{\mathcal{B}})^{\alpha}\phi=
\frac{\Gamma(\alpha)\Gamma(m-\alpha)}{\Gamma(m)}
\int_{0}^{\infty}
\lambda^{\alpha-1}
[(-(S_{\mu})_{\mathcal{B}})
(\lambda-(S_{\mu})_{\mathcal{B}})^{-1}]^{m}
\phi\:\:d\lambda.
\end{equation}

By definition of $\mathcal{B}$ and the fact that $L^{1}(\Rnm,sr)\cap L^{\infty}(\Rnm,r)\subset L^{p}(\Rnm,sr^{p})$ for all $1\leq p\leq\infty$ then $\mathcal{B}\subset D(S_{\mu,p})$ for all $1\leq p\leq \infty$, in particular, $\mathcal{B}\subset D(S_{\mu,2})$. Then from  Propositions 8.3 and 8.4 in \cite{MT07} we obtain that:
\begin{equation}\label{eqpot2}
(-S_{\mu,2})^{\alpha}\phi=
\frac{\Gamma(\alpha)\Gamma(m-\alpha)}{\Gamma(m)}
\int_{0}^{\infty}
\lambda^{\alpha-1}
[-S_{\mu,2}(\lambda-(S_{\mu,2})^{-1}]^{m}
\phi\:\:d\lambda.
\end{equation}

Since for $\phi\in\mathcal{B}$, the integrating into the expressions are equal and the fact that the convergence in  $\mathcal{B}$ implies the convergence in  $L^{2}(\Rnm)$ (see Lemma 2.1 and Remark 5.3 in \cite{Mo18}) we obtain the equality of  (\ref{eqpot1}) and (\ref{eqpot2}) as functions.

\bigskip
We conclude that
\[
(-(S_{\mu})_{\mathcal{B}})^{\alpha}\phi=
\HZ_{\mu }\|y\|^{2\alpha}\HZ_{\mu}\phi,
\quad \phi\in\mathcal{B},
\]
(see \cite[Proposition 8.4 ]{Mo18}). From the last equality and \eqref{eqequal0}, we have that

\begin{equation} \label{eqpot3}
((-(S_{\mu})_{\mathcal{B'}})^{\alpha}u,\phi)=
(u,\HZ_{\mu }\|y\|^{2\alpha}\HZ_{\mu }\phi)=0,
\end{equation}
for all $\phi\in\mathcal{B}$.

Since $\mathcal{B'}\subset\mathcal{H'}_{\mu} $  (see \cite[Remark 6.2]{Mo18}), we can consider the Hankel transform in $\mathcal{B'}$. We are going to see that the following affirmation holds:

\medskip
\begin{center}
" \sl{If  $u\in\mathcal{B'}$ is such that \eqref{eqpot3} is verified, then $(\HZ_{\mu}u,\psi)=0$ for all $\psi\in \mathcal{H}_{\mu}$ such that $\supp\psi\subset \Rnm\cap\{x:\|x\|\geq a\}$ with  $a>0$.}"
\end{center}
\medskip

Let $u\in\mathcal{B'}$ such that \eqref{eqpot3} is valid and $\psi\in \mathcal{H}_{\mu}$ such that $\supp\psi\subset \Rnm\cap\{x:\|x\|\geq a\}$ with  $a>0$. Then, by Lemma \ref{lemmamultpot},  $\|x\|^{-2\alpha}\psi(x)\in\mathcal{H}_{\mu}$ and  since the Hankel transform is an isomorphism in $\mathcal{H}_{\mu}$, there exists $\phi\in\mathcal{H}_{\mu}$ such that $\HZ_{\mu}\phi=\|x\|^{-2\alpha}\psi(x)$. So,

\[(\HZ_{\mu}u,\psi)=
(\HZ_{\mu}u,\|x\|^{2\alpha}\|x\|^{-2\alpha}\psi)=
(\HZ_{\mu}u,\|x\|^{2\alpha}\HZ_{\mu}\phi)=
(u,\HZ_{\mu}\|x\|^{2\alpha}\HZ_{\mu}\phi).
\]

Consequently, from  \eqref{eqpot3} we conclude that $(\HZ_{\mu}u,\psi)=0$, then  the assertion is valid. Thus by \cite[Theorem 4.1]{GMQ18}, there exist  $N\in\mathbb{N}_{0}$ and scalars $c_{k}$ with $|k|<N$  such that   $\HZ_{\mu}u=\sum_{|k|<N}c_{k}S^{k}_{\mu}\delta_{\mu}$ where $\delta_{\mu}$ is given by \cite[equation (2.3)]{GMQ18} for $k=0$. Then,
\[
u=x^{\mu+\frac{1}{2}}
\sum_{|k|\leq N}c_{k}(-1)^{|k|}\|x\|^{2k}
\]
\end{proof}

\begin{remark}[Regular distributions in $\mathcal{B'}$]
\end{remark}

If $f\in L_{loc}^{1}(\Rnm)$ and $f=O(x^{\mu+\frac{1}{2}})$ then $f$ is a regular distribution in  $\mathcal{B'}$ given by
\[
(f,\phi)=
\int_{\Rnm} f(x)\phi(x)\:dx,
\quad \phi\in\mathcal{B},
\]
and
\begin{align*}
|(f,\phi)|&=
\Bigl|\int_{\Rnm} f(x)\phi(x)\:dx\Bigr|\\
&\leq
\Bigl|\int_{\|x\|\leq M} f(x)\phi(x)dx\Bigr|+
\Bigl|\int_{\|x\|\geq M} f(x)\phi(x)dx\Bigr|\\
&\leq
\int_{\|x\|\leq M} |r^{-1}(x)f(x)|dx \:\|\phi\|_{L^{\infty}(r)}
+
\int_{\|x\|\geq M} c x^{\mu+\frac{1}{2}}|\phi(x)|dx\\
&=
C\|\phi\|_{L^{\infty}(r)}+c\:C_{\mu}\|\phi\|_{L^{1}(rs)}\leq C' \rho_{0}(\phi). \end{align*}

\medskip
\begin{corollary}
If $f\in L_{loc}^{1}(\Rnm)$,  $f=O(x^{\mu+\frac{1}{2}})$ and $(-(S_{\mu})_{\mathcal{B'}})^{\alpha}f=0$ then $f=C\;x^{\mu+\frac{1}{2}}$.
\end{corollary}

\section{Distributional Liouville  theorem for \texorpdfstring{$(-(\Delta_{\mu}))^{\alpha}$}{Deltamualpha}.}

From theory of similar operators given in  \cite{Mo18}, by the similarity of $S_{\mu}$ and  $\Delta _{\mu}$, and by the non-negativity of the part of $-S_{\mu}$ in $L^{1}(\Rnm,sr)$ and  $L^{\infty}(\Rnm,r)$ we deduce the non-negativity of the part of $-\Delta_{\mu}$ in $L^{1}(\Rnm,s)$ and $L^{\infty}(\Rnm)$. Consequently, we infer the non-negativity of the part of $-\Delta_{\mu}$ in the Banach space $Z=L^{1}(\Rnm,s)\cap L^{\infty}(\Rnm)$ with norm
\begin{equation*}
\left\| f\right\|_{Z}=
\max\left\{ \left\| f\right\|_{L^{1}(\Rnm,s)},
\left\| f\right\| _{L^{\infty }(\Rnm)}\right\}.
\end{equation*}

\bigskip
\noindent Thus, if we consider $Y$ as in Section 5 and  $L_{r}:Y\to Z$ given by $L_{r}f=rf$ then
\begin{align*}
\left\| rf\right\|_{Z}
&=
\max \left\{ \left\| rf\right\|_{L^{1}(\Rnm,s)},
     \left\| rf\right\|_{L^{\infty }(\Rnm)}\right\}\\
&=
\max \left\{ \left\| f\right\|_{L^{1}(\Rnm,rs)},
     \left\| f\right\|_{L^{\infty }(\Rnm,r)}\right\}\\
&=
\left\| f\right\|_{Y},
\end{align*}
so, $L_{r}$ is an isometric isomorphism.

\vspace{.2in}
Moreover, we can consider the locally convex space $\mathcal{F}$ given by:
\[
\mathcal{F}=
\{f\in Z: (\Delta_{\mu})^{k}f\in Z
\quad\mbox {for}\quad k=0,1,2,\cdots\}
=
\bigcap_{k=0}^{\infty}D\bigl[((\Delta_{\mu})_{Z})^{k}\bigr],\]

\noindent where with $(\Delta_{\mu})_{Z}$ we denote the part of $\Delta_{\mu}$ in $Z$. The space $\mathcal{F}$ is endowed with the topology generated by the family of seminorms given by
\[
\gamma_{m}(f)=
\max_{0\leq k\leq m}
\{\| (\Delta_{\mu})^{k}f \|_{Z}\},
\quad m=0,1,2,\cdots
\]

Thus, the space $\mathcal{F}$ verifies that is a Fréchet space and from Remarks \ref{Remark4.8 bis} and  \ref{Remark4.9} we deduce that  $\mathcal{F}\subset  C_{0}(\Rnm)$, $\mathcal{F}\subset L^{p}(\Rnm,s)$ for all $1\leq p<\infty$, $\mathcal{F}\subset L^{\infty}(\Rnm)$, $\mathcal{H}_{\mu}\subset \mathcal{F}$ and the topology of $\mathcal{H}_{\mu}$ induced by $\mathcal{F}$ is weaker than the usual topology in $\mathcal{H}_{\mu}$. Moreover, the operator $\Delta_{\mu}$ verifies that
\[
\gamma_{m}(\Delta_{\mu}f)=
\max_{0\leq k\leq m}
\{\| (\Delta_{\mu})^{k}\Delta_{\mu}f \|_{Z}\}
=\gamma_{m+1}(f),
\]

\noindent for all $f\in \mathcal{F}$. Then $(\Delta_{\mu})_{\mathcal{F}}$ the part  of  $\Delta_{\mu}$ in $\mathcal{F}$, is a continuous
\[
(\Delta_{\mu})_{\mathcal{F}}:\mathcal{F}\to\mathcal{F}.
\]

\noindent If $f\in\mathcal{B}$, (see \eqref{eqB}), then $rf\in\mathcal{F}$ and

\begin{align*}
\gamma_{m}(rf)&=
\max_{0\leq k\leq m}
\{\| (\Delta_{\mu})^{k}rf \|_{Z}\}
=
\max_{0\leq k\leq m}
\{\| r(S_{\mu})^{k}r^{-1}rf \|_{Z}\}\\
&=
\max_{0\leq k\leq m}
\{\| r(S_{\mu})^{k}f \|_{Z}\}
=\max_{0\leq k\leq m}
\{\| (S_{\mu})^{k}f \|_{Y}\}\\
&=\rho_{m}(f),
\end{align*}

\noindent where we have consider \eqref{similaridad Op de Bessel}. So, the application $L_{r}:\mathcal{B}\to \mathcal{F}$, given by $L_{r}f=rf$ is an isomorphism of locally convex spaces with inverse given by  $L_{r^{-1}}:\mathcal{F}\to \mathcal{B}$.

\begin{remark} Since $\mathcal{B}$ and $\mathcal{F}$ are isomorphic then we can deduce that $\mathcal{F'}$ is sequentially complete as $\mathcal{B'}$ is also sequentially complete (see Proposition \ref{B' secuencialmente completo}).
\end{remark}

So, if we consider the continuous operator $(S_{\mu})_{\mathcal{B}}:\mathcal{B}\to \mathcal{B}$,  then by \eqref{similaridad Op de Bessel} we obtain the similarity relation,

\begin{equation}\label{relopBesselHir}
(\Delta_{\mu})_{\mathcal{F}}=
L_{r}\:(S_{\mu})_{\mathcal{B}}\:L_{r^{-1}}.
\end{equation}

\noindent We deduce by \eqref{relopBesselHir} the non-negativity of $-(\Delta_{\mu})_{\mathcal{F}}$ and by \cite[Proposition 1.1]{Mo18}, for $\alpha\in\C$, $\Re\alpha>0$, we have that

\begin{equation}\label{relopBessHirF}
(-(\Delta_{\mu})_{\mathcal{F}})^{\alpha}=L_{r}\:(-(S_{\mu})_{\mathcal{B}})^{\alpha}\:L_{r^{-1}}.
\end{equation}

\noindent Consequently,
\[
(-(\Delta_{\mu})_{\mathcal{F'}})^{\alpha}=
((-(\Delta_{\mu})_{\mathcal{F}})^{*})^{\alpha}=
((-(\Delta_{\mu})_{\mathcal{F}})^{\alpha})^{*},
\]
where we have considered in the second equality that $\mathcal{F}$ is a Fréchet space, (see \cite[pp.134]{MS01}). Thus,

\begin{equation}\label{eqpotBessHirdistr1}
(-(\Delta_{\mu})_{\mathcal{F'}})^{\alpha}=
(L_{r^{-1}})^{*}\:((-(S_{\mu})_{\mathcal{B}})^{\alpha})^{*}\:(L_{r})^{*}=
(L_{r^{-1}})^{*}\:(-(S_{\mu})_{\mathcal{B'}})^{\alpha}\:(L_{r})^{*},
\end{equation}
and for  $T\in\mathcal{F'}$, $\phi\in\mathcal{F}$

\begin{align}
((-(\Delta_{\mu})_{\mathcal{F'}})^{\alpha}T,\phi)
&=
((L_{r^{-1}})^{*}\:(-(S_{\mu})_{\mathcal{B'}})^{\alpha}\:(L_{r})^{*}T,\phi)\nonumber\\
&=(T,L_{r}(-(S_{\mu})_{\mathcal{B}})^{\alpha}L_{r^{-1}}\phi).\label{eqpotBessHirdistr2}
\end{align}

From now on, we will use the notation:

\[(\Delta_{\mu})_{\mathcal{F}}=
x^{-\mu-\frac{1}{2}}\:
(S_{\mu})_{\mathcal{B}}\:
x^{\mu+\frac{1}{2}},
\]
\[(-(\Delta_{\mu})_{\mathcal{F}})^{\alpha}=
x^{-\mu-\frac{1}{2}}\:
(-(S_{\mu})_{\mathcal{B}})^{\alpha}\:
x^{\mu+\frac{1}{2}}\]
and
\[(-(\Delta_{\mu})_{\mathcal{F'}})^{\alpha}=
x^{\mu+\frac{1}{2}}\:
(-(S_{\mu})_{\mathcal{B'}})^{\alpha}\:
x^{-\mu-\frac{1}{2}}\]
to refer to  \eqref{relopBesselHir}, \eqref{relopBessHirF} and \eqref{eqpotBessHirdistr1}. In the last equation, the operators $x^{\mu+\frac{1}{2}}$ and $x^{-\mu-\frac{1}{2}}$ represent  $(L_{r^{-1}})^{*}$ and $(L_{r})^{*}$, so,

\[x^{\mu+\frac{1}{2}}:\mathcal{B'}\to \mathcal{F'},\]
\[x^{-\mu-\frac{1}{2}}:\mathcal{F'}\to\mathcal{B'},\]

\noindent are given by
\[(x^{\mu+\frac{1}{2}}T_{1},\phi)=
(T_{1},x^{\mu+\frac{1}{2}}\phi),
\qquad (T_{1}\in \mathcal{B'}), (\phi\in\mathcal{F})\]
\[(x^{-\mu-\frac{1}{2}}T_{2},\psi)=
(T_{2},x^{-\mu-\frac{1}{2}}\psi),
\qquad (T_{2}\in \mathcal{F'}), (\psi\in\mathcal{B})\]

\vspace{.2in}
Now we are able to establish the following theorem:

\begin{theorem}
Let $u\in\mathcal{F'}$ and $\alpha\in\C$ with $\Re\alpha>0$. If $(-(\Delta_{\mu})_{\mathcal{F'}})^{\alpha}u=0$ then there exists a polynomial $p$ such that $u=x^{2\mu+1}p(\|x\|^{2})$.
\end{theorem}

\begin{proof}
Let  $u\in \mathcal{F'}$ such that     $(-(\Delta_{\mu})_{\mathcal{F'}})^{\alpha}u=0$. Then
\begin{equation}\label{eqequal1}
((-(\Delta_{\mu})_{\mathcal{F'}})^{\alpha}u,\phi)=
(x^{\mu+\frac{1}{2}}\:
(-(S_{\mu})_{\mathcal{B'}})^{\alpha}\:
x^{-\mu-\frac{1}{2}}u,\phi)=0
\end{equation}
for all $\phi\in\mathcal{F}$. Since   $(-(S_{\mu})_{\mathcal{B'}})^{\alpha}\:x^{-\mu-\frac{1}{2}}u \in \mathcal{B'}$, then given $\psi \in\mathcal{B}$ and considering \eqref{eqequal1}, we obtain that
\[((-(S_{\mu})_{\mathcal{B'}})^{\alpha}\:x^{-\mu-\frac{1}{2}}u,\psi)=
((-(S_{\mu})_{\mathcal{B'}})^{\alpha}\:x^{-\mu-\frac{1}{2}}u,x^{\mu+\frac{1}{2}}x^{-\mu-\frac{1}{2}}\psi)=\]
\[(x^{\mu+\frac{1}{2}}(-(S_{\mu})_{\mathcal{B'}})^{\alpha}\:x^{-\mu-\frac{1}{2}}u,x^{-\mu-\frac{1}{2}}\psi)=0.\]

\noindent By Theorem \ref{Teorema5.1} we deduce that there exists a polynomial $p$ such that $x^{-\mu-\frac{1}{2}}u=x^{\mu+\frac{1}{2}}p(\|x\|^{2})$ and consequently $u=x^{2\mu+1}p(\|x\|^{2})$.

\smallskip
\end{proof}

\begin{corollary}
If $f\in L_{loc}^{1}(\Rnm)$,  $f=O(x^{2\mu+1})$ and $(-(S_{\mu})_{\mathcal{B'}})^{\alpha}f=0$ then $f=C\;x^{2\mu+1}$.
\end{corollary}

\clearpage
\appendix

\section{Special functions}

In this appendix we summarize properties of the Bessel function of the first species and order $\alpha$ given by:

\begin{equation}\label{Ap-bessel}
J_{\alpha}(z)=\left(\frac{z}{2}\right)^{\alpha}
\sum_{n=0}^{\infty} \frac{(-1)^{n}\,
\left (\frac{z}{2}\right)^{2n}}{n!\,\Gamma(\alpha+n+1)}.
\end{equation}

\noindent According to  \cite[p.310]{Hi60}, for $\alpha\in\R$,  $\alpha>-\frac{1}{2}$, the Bessel function verifies that

\begin{equation}\label{cota Bessel}
|C_{\alpha}\,z^{-\alpha}\,J_{\alpha}(z)|\leq 1,
\end{equation}

\noindent where $C_{\alpha}=2^{\alpha}\Gamma(\alpha+1)$

\medskip
\noindent The following equalities are also verified:

\begin{equation}\label{Cilindrícas}
\int_{0}^{\pi} \frac{J_{\alpha}\left(\sqrt{y^{2}+z^{2}-2yz\cos\phi}\right)}{\left(y^{2}+z^{2}-2yz\cos\phi\right)^{\frac{1}{2}\alpha}} \sin^{2\alpha}\phi\: d\phi =2^{\alpha}\,\Gamma(\alpha+1/2)\,\Gamma(1/2)\,\frac{J_{\alpha}(y)}{y^{\alpha}}\,\frac{J_{\alpha}(z)}{z^{\alpha}},
\end{equation}

\noindent for $\alpha>-\frac{1}{2}$,  see \cite[pp.367]{Wa} and

\begin{equation}\label{Prop3.14-eq1}
 \int_{0}^{\infty}e^{\frac{-ay^{2}}{2}}\,J_{\alpha}(ry)\,
 y^{\alpha+1}\,dy=
 r^{\alpha}\,a^{-\alpha-1}\,e^{-\frac{r^{2}}{2a}},
\end{equation}

\noindent for $\alpha>-1$ and  $a>0$, see \cite[pp.46]{Ob}.

\bigskip

\noindent The following equalities are valid for integrals that involve the Gamma function: Let $a>0$ and $\mu>-1$, then the following equalities are valid
\begin{equation}\label{ApA-eq1}
\int_{0}^{\infty} e^{-\frac{x^{2}}{2}}\,x^{2\mu+1}\,dx=
2^{\mu}\,\Gamma(\mu+1).
\end{equation}

\begin{equation}\label{ApA-eq2}
\int_{0}^{\infty} e^{-\frac{x^{2}}{2a}}\,x^{2\mu+1}\,dx=
2^{\mu}\,\Gamma(\mu+1)\,a^{\mu+1}.
\end{equation}

\begin{equation}\label{integral sin}
\int_{0}^{\pi/2}\sin^{2r}\theta\;d\theta=
\frac{\Gamma(1/2)\,\Gamma(r+1/2)}{2\Gamma (r+1)}=
\frac{\sqrt{\pi}\,\Gamma(r+1/2)}{2\Gamma (r+1)}.
\end{equation}

\bigskip
\noindent Another important equation is the Euler Complements Formula
\begin{equation}\label{complementosEuler}
\Gamma(\alpha)\Gamma(1-\alpha)=
\frac{\pi}{\sin \pi\alpha}
\quad (0<\Re\alpha<1).
\end{equation}

\vspace{.2in}

\section{ Some results on Hankel transforms, convolution and the Inversion theorem}

\begin{proof}[Proof of Proposition \ref{prop-nDZ}]\label{Proof-prop-nDZ}
Assertion $(i)$ follows immediately.\\

To proof $(ii)$ let us see the integrability of  $\DZ_{\mu}(x,y,z)
\prod\limits_{i=1}^{n}
\{\sqrt{z_it_i}J_{\mu_i}(z_it_i)\}$. First, we observe that
\[
\int_{\Rnm}\DZ_{\mu}(x,y,z)
\prod\limits_{i=1}^{n}
\{\sqrt{z_it_i}J_{\mu_i}(z_it_i)\}\;dz=
\int_{\Rnm}
\prod\limits_{i=1}^{n}
\{\DZ_{\mu_i}(x_i,y_i,z_i)\sqrt{z_it_i}J_{\mu_i}(z_it_i)\}\;dz
\]
To proof this, it will be enough to see that
\[
\int_{0}^{\infty}
D_{\alpha}(u,v,w)\sqrt{wt}J_{\alpha}(wt) \;dw
=
t^{-\alpha-1/2}
\sqrt{ut}
\,J_{\alpha}(ut)
\sqrt{vt}
\,J_{\alpha}(vt),
\]
\noindent where $u, v, w$ and $t\in(0,\infty)$ and $\alpha>-\frac{1}{2}$.  Since

\clearpage
\begin{align*}
|D_{\alpha} & (u,v,w)\sqrt{wt}J_{\alpha}(wt)|
=
|D_{\alpha}(u,v,w) (wt)^{\alpha+1/2} (wt)^{-\alpha}J_{\alpha}(wt)|\\
&\leq
C |D_{\alpha}(u,v,w) (wt)^{\alpha+1/2}|\\
&\leq
C t^{\alpha+1/2}(uv)^{-\alpha+1/2} w [(u+v)^{2}-w^2]^{\alpha-1/2} [w^{2}-(u-v)^2]^{\alpha-1/2}.
\end{align*}

\noindent Here we have used that the function $w^{-\alpha}J_{\alpha}(w)$ is bounded for $\alpha\in\R$, greater than $-\frac{1}{2}$. In fact, $|w^{-\alpha}J_{\alpha}(w)|\leq \frac{\sqrt{\pi}}{2^{\alpha}\Gamma(\alpha+1/2)}$. Moreover $\supp A(u,v,w)\subset[|u-v|,u+v]$, so

\[
\int_{0}^{\infty}
|D_{\alpha}(u,v,w)\sqrt{wt}J_{\alpha}(wt)|\:dw
\leq C
\int_{|u-v|}^{u+v} w [(u+v)^{2}-w^2]^{\alpha-1/2} [w^{2}-(u-v)^2]^{\alpha-1/2} \;dw,
\]

\noindent  rewriting the last integral
\[
\int_{|u-v|}^{u+v}
\frac{w}{[(u+v)^{2}-w^2]^{1/2-\alpha} [w^{2}-(u-v)^2]^{1/2-\alpha}} \;dw=
\]

\begin{equation}\label{prop-nDZ.eq1}
\int_{|u-v|}^{u+v}
\frac{w}{[(u+v)-w]^{1/2-\alpha}[(u+v)+w]^{1/2-\alpha}
         [w-|u-v|]^{1/2-\alpha} [w+|u-v|]^{1/2-\alpha}} \;dw.
\end{equation}

\noindent  To analyze integrability, we separate the region of integration $[|u-v|, u+v]$ considering $c$, arbitrary and fix such that $|u-v|\leq c\leq u+v$. Let us note that it is possible to write \eqref{prop-nDZ.eq1} as

\begin{equation}\label{prop-nDZ.eq2}
\int_{|u-v|}^{c} \frac{f_{1}(w)}{[w-|u-v|]^{1/2-\alpha}} dw +
\int_{c}^{u+v} \frac{f_{2}(w)}{[(u+v)-w]^{1/2-\alpha}} dw
\end{equation}

\noindent where $f_{1}(w)=\frac{w}{[(u+v)^{2}-w^{2}]^{1/2-\alpha}[w+|u-v|]^{1/2-\alpha}}$ and  $f_{2}(w)=\frac{w}{[(u+v)+w]^{1/2-\alpha}[w^{2}-(u-v)^{2}]^{1/2-\alpha}}$.

\smallskip
\noindent Since $f_{1}(w)$ is a continuous function in  $[|u-v|,c]$, it results bounded. Then there is a constant $C_{1}>0$ such that $|f_{1}(w)|\leq C_{1}$ for all  $w\in [|u-v|,c]$. The same goes for $f_{2}$ in $[c, u+v]$. The problem is then reduced to studying the integrability of $\frac{1}{[w-|u-v|]^{1/2-\alpha}}$ in $[|u-v|,c]$ and  $\frac{1}{[(u+v)-w]^{1/2-\alpha}}$ in $[c,u+v]$,

\begin{equation}\label{prop-nDZ.eq3}
C_{1}\int_{|u-v|}^{c} \frac{1}{[w-|u-v|]^{1/2-\alpha}} dw +
C_{2}\int_{c}^{u+v} \frac{1}{[(u+v)-w]^{1/2-\alpha}} dw.
\end{equation}

\noindent Since $\alpha>-\frac{1}{2}$ both integrals in \eqref{prop-nDZ.eq3} are finite. Then it is proved that $D_{\alpha}(u,v,w)\sqrt{wt}J_{\alpha}(wt)$ is integrable in $(0,\infty)$ for $\alpha>-\frac{1}{2}$.

\bigskip
Now, we consider the change of variables
\begin{equation}
T:(0,\pi)\to (0,\infty)
\quad \text{con}\quad
T(\theta)=\sqrt{u^{2}+v^{2}-2uv \cos\theta},
\end{equation}
where $\quad\frac{d}{d\theta}T(\theta)>0,
       \quad T(0)=|u-v|,
       \quad T(\pi)=u+v\quad$ and

\[
\frac{d}{d\theta}T(\theta)
=\frac{uv\sin\theta}{\sqrt{u^{2}+v^{2}-2xy\cos\theta}}.
\]
So,

\begin{align*}
\int_{0}^{\infty}
&
D_{\alpha}(u,v,w)\sqrt{wt}J_{\alpha}(wt) \;dw
=
\frac{2^{\alpha-1}(xy)^{-\alpha+1/2}}{\Gamma(\alpha+1/2)\sqrt{\pi}}
\int_{|u-v|}^{u+v} A(u,v,w)^{2\alpha-1}
w^{-\alpha+1/2}\sqrt{wt}J_{\alpha}(wt) \;dw\\
&=
\frac{2^{\alpha-1}(uv)^{-\alpha+1/2}\sqrt{t}}{\Gamma(\alpha+1/2)\sqrt{\pi}}
\int_{|u-v|}^{u+v} A(u,v,w)^{2\alpha-1}
w^{-\alpha+1}J_{\alpha}(wt) \;dw\\
&=
\frac{2^{\alpha-1}(uv)^{-\alpha+1/2}\sqrt{t}}{\Gamma(\alpha+1/2)\sqrt{\pi}}
\int_{0}^{\pi}
\frac{
\left(\frac{uv}{2}\sin\theta\right)^{2\alpha-1}
J_{\alpha}(\sqrt{u^{2}+v^{2}-2uv\cos\theta}\;t)\:
uv\sin\theta
}
{(\sqrt{u^{2}+v^{2}-2uv\cos\theta})^{\alpha-1}\sqrt{u^{2}+v^{2}-2uv\cos\theta}}
\;d\theta  \\
&=
\frac{2^{-\alpha}(uv)^{\alpha+1/2}\sqrt{t}}{\Gamma(\alpha+1/2)\sqrt{\pi}}
\int_{0}^{\pi}
\frac{J_{\alpha}(\sqrt{u^{2}+v^{2}-2uv\cos\theta}\;t)}
{(\sqrt{u^{2}+v^{2}-2uv\cos\theta})^{\alpha}}
\:\sin^{2\alpha}\theta
\;d\theta\\
&=
\frac{2^{-\alpha}(uv)^{\alpha+1/2}\sqrt{t}\;t^{\alpha}}{\Gamma(\alpha+1/2)\sqrt{\pi}}
\int_{0}^{\pi}
\frac{J_{\alpha}(\sqrt{(ut)^{2}+(vt)^{2}-2(ut)(vt)\cos\theta})}
{\left(\sqrt{(ut)^{2}+(vt)^{2}-2(ut)(vt)\cos\theta}\right)^{\alpha}}
\:\sin^{2\alpha}\theta
\;d\theta  \\
&=
\frac{2^{-\alpha}(uv)^{\alpha+1/2}\sqrt{t}\;t^{\alpha}}
{\Gamma(\alpha+1/2)\sqrt{\pi}}
2^{\alpha}\,\Gamma(\alpha+1/2)\,\Gamma(1/2)
\,\frac{J_{\alpha}(ut)}{(ut)^{\alpha}}
\,\frac{J_{\alpha}(vt)}{(vt)^{\alpha}}\\
&=t^{-\alpha-1/2}
\sqrt{ut}
\,J_{\alpha}(ut)
\sqrt{vt}
\,J_{\alpha}(vt)
\end{align*}

\noindent where have we used \eqref{Cilindrícas} for $\alpha>-\frac{1}{2}$.

\medskip
To proof $(iii)$, let us observe that
\[
\int_{\Rnm} z^{\mu+1/2}\DZ_{\mu}(x,y,z) dz=
\int_{\Rnm}
\prod_{i=1}^{n} \{z_{i}^{\mu_{i}+1/2}D_{\mu_i}(x_i,y_i,z_i)\} dz_1\ldots dz_n.
\]

\noindent Suffice it to see then that
\[
\int_{0}^{\infty} w^{\alpha+1/2} D_{\alpha}(u,v,w)\;dw =
C_{\alpha}^{-1} u^{\alpha+1/2}  v^{\alpha+1/2},
\]
\noindent where $u,v,w\in(0,\infty)$, $\alpha>-\frac{1}{2}$. Then

\begin{align*}
&\int_{0}^{\infty} w^{\alpha+1/2} D_{\alpha}(u,v,w)\;dw =
\frac{2^{\alpha-1}}{\Gamma(\alpha+1/2)\sqrt{\pi}}
(uv)^{-\alpha+1/2}
\int_{0}^{\infty}w\;A(u,v,w)^{2\alpha-1}\;dw\\
&=
\frac{2^{\alpha-1}(uv)^{-\alpha+1/2}}{\Gamma(\alpha+1/2)\sqrt{\pi}}
\int_{0}^{\pi}\sqrt{u^2+v^2-2uv\cos\theta}
\left(\frac{uv\sin\theta}{2}\right)^{2\alpha-1}
\frac{uv\sin\theta}{\sqrt{u^2+v^2-2uv\cos\theta}} d\theta\\
&=
\frac{2^{-\alpha}(uv)^{\alpha+1/2}}{\Gamma(\alpha+1/2)\sqrt{\pi}}
\int_{0}^{\pi}\sin^{2\alpha}\theta\;d\theta
=
\frac{2^{-\alpha}(uv)^{\alpha+1/2}}{\Gamma(\alpha+1)}
=
c_{\alpha}^{-1} u^{\alpha+1/2} v^{\alpha+1/2},
\end{align*}
\noindent where we ave used \eqref{integral sin} from Appendix.
\end{proof}

\begin{proof}[Proof of Lemma \ref{lema3.3} (i)] \label{Proof-lema3.3}
Let $f\in L^{1}(sr)$ and $g\in L^{\infty}(r)$, let us  see that $f\convZ g\in L^{\infty}(r)$. Since

\begin{equation}\label{lema3.3-eq3}
f\convZ g(x)=
\int_{\Rnm}\int_{\Rnm}f(y)g(z)\:\DZ_{\mu}(x,y,z)\;dy\:dz,
\end{equation}

\noindent then
\begin{equation}\label{lema3.3-eq4}
|f\convZ g(x)|\leq\int_{\Rnm}\int_{\Rnm}|f(y)|\:|g(z)|\:\DZ_{\mu}(x,y,z)\:s(y)s(z)\:dy\:dz.
\end{equation}

\noindent Let us analyze the existence of the next iterated integral
\begin{align}
&\int_{\Rnm}\left\{
\int_{\Rnm}
|f(y)|\:|g(z)|\:\DZ_{\mu}(x,y,z)\:dz\right\}\:dy\nonumber\\
& =
\int_{\Rnm}|f(y)|
\left\{\int_{\Rnm}
|g(z)|\:\DZ_{\mu}(x,y,z)\:dz\right\}\:dy\nonumber\\
& =
\int_{\Rnm}|f(y)|
\left\{\int_{\Rnm}
|r(z)g(z)|\:r^{-1}(z)\DZ_{\mu}(x,y,z)\:dz\right\}\:dy\nonumber\\
&\leq \|g\|_{L^{\infty}(r)}
\int_{\Rnm}|f(y)|
\left\{\int_{\Rnm}
z^{\mu+1/2}\DZ_{\mu}(x,y,z)\:dz\right\}\:dy\nonumber\\
&= \|g\|_{L^{\infty}(r)}\int_{\Rnm}
|f(y)|\:C_{\mu}^{-1} x^{\mu+1/2} y^{\mu+1/2}\:dy\nonumber\\
&= \|g\|_{L^{\infty}(r)} x^{\mu+1/2}
\int_{\Rnm}|f(y)|s(y)r(y)\:dy=
 x^{\mu+1/2}\|g\|_{L^{\infty}(r)}\|f\|_{L^{1}(sr)}\label{lema3.3-eq5}
\end{align}

Since the integral
$\int_{\Rnm}|f(y)|
\left\{\int_{\Rnm}
|g(z)|\:\DZ_{\mu}(x,y,z)\:dz\right\}\:dy<\infty$ for  $f\in L^{1}(sr)$ and  $g\in L^{\infty}(r)$, then Tonelli's Theorem allows us to affirm that the integral in \eqref{lema3.3-eq3} exists for all $x\in\Rnm$, and from \eqref{lema3.3-eq4} and  \eqref{lema3.3-eq5} the desired result is obtained.

\begin{proof}[(ii)] Let
\begin{equation}\label{lema3.3-eq6}
K(x,z)=
C_{\mu}x^{-\mu-1/2}z^{-\mu-1/2}
\int_{\Rnm} f(y)\:\DZ_{\mu}(x,y,z)\;dy
\end{equation}

\begin{align*}
\int_{\Rnm} & |K(x,z)|\;s(x)\;dx =
\int_{\Rnm}\left|
C_{\mu}x^{-\mu-1/2}z^{-\mu-1/2}
\int_{\Rnm} f(y)\:\DZ_{\mu}(x,y,z)\;dy
\right|\frac{x^{2\mu+1}}{C_{\mu}} dx\\
&\leq
\int_{\Rnm}z^{-\mu-1/2}
\left\{\int_{\Rnm}
|f(y)|\;\DZ_{\mu}(x,y,z)\;dy
\right\} x^{\mu+1/2} dx\\
&=
\int_{\Rnm}\left\{
\int_{\Rnm} z^{-\mu-1/2}|f(y)|x^{\mu+1/2}\DZ_{\mu}(x,y,z)\;dy
\right\}dx\\
&=
\int_{\Rnm}\left\{
\int_{\Rnm} z^{-\mu-1/2}|f(y)|x^{\mu+1/2}\DZ_{\mu}(x,y,z)\;dx
\right\}dy\\
& =
\int_{\Rnm}
z^{-\mu-1/2}|f(y)|\left\{
\int_{\Rnm} x^{\mu+1/2}\DZ_{\mu}(x,y,z)\;dx\right\}dy\\
& =
\int_{\Rnm}
z^{-\mu-1/2}|f(y)|\:C_{\mu}^{-1} y^{\mu+1/2}z^{\mu+1/2} dy
= \|f\|_{L^{1}(sr)}<\infty
\end{align*}
Thus
\begin{equation}\label{lema3.3-eq9}
\int_{\Rnm}|K(x,z)|\;s(x)\;dx\leq \|f\|_{L^{1}(sr)}=\|rf\|_{L^{1}(s)}
\end{equation}
and similarly
\begin{equation}\label{lema3.3-eq10}
\int_{\Rnm}|K(x,z)|\;s(z)\;dz\leq \|f\|_{L^{1}(sr)}=\|rf\|_{L^{1}(s)}.
\end{equation}
So, from \cite[Theorem 6.18 - pp.193]{Fo99} if  $h\in L^{p}(s)$ then \[
Th(x)=\int_{\Rnm}h(z)\;K(x,z)\;s(z)\;dz
\]
\noindent exists for almost every $x\in\Rnm$ and
\[
\|Th\|_{L^{p}(s)}\leq \|rf\|_{L^{1}(s)} \|h\|_{L^{p}(s)}
\]
In particular, for $h=rg$, since $g\in L^{p}(sr^p)$
\[
\|g\|_{L^{p}(sr^p)}^{p}=
\int_{\Rnm}|g(z)|^{p}s(z)r^{p}(z)\;dz=
\int_{\Rnm}|r(z)g(z)|^{p}s(z)(z)\;dz=
\|rg\|_{L^{p}(s)},
\]
\noindent then $h\in L^{p}(s)$.

\begin{align*}
T(h)(x)=T(rg)(x)
&=
\int_{\Rnm}\int_{\Rnm}
r(z)g(z)\:c_{\mu}x^{-\mu-1/2}x^{-\mu-1/2}
f(y)\DZ_{\mu}(x,y,z)\:s(z)\:dy\:dz\\
&=x^{-\mu-1/2}\int_{\Rnm}\int_{\Rnm}
f(y)\:g(z)\:\DZ_{\mu}(x,y,z)\:dy\:dz\\
&=x^{-\mu-1/2} f\convZ g(x)
\end{align*}

\[
\|f\convZ g\|_{L^{p}(sr^p)}=
\|r(f\convZ g)\|_{L^{p}(s)}=
\|T(rg)\|_{L^{p}(s)}\leq
\|rf\|_{L^{1}(s)} \|rg\|_{L^{p}(s)}=
\|f\|_{L^{1}(sr)} \|g\|_{L^{p}(sr^p)}
\]
\phantom{\qedhere}
\end{proof}

\end{proof}

\vspace{.2in}
Hirschman defined in  \cite{Hi60} for the $1$-dimensional case a kernel  $\mathfrak{D}_{\alpha}$ which is defined for  $u,v,w\in (0,\infty)$, $\alpha>-\frac{1}{2}$, by
\begin{equation}\label{DH}
\mathfrak{D}_{\alpha}(u,v,w)=
\frac{2^{3\alpha-1}\Gamma^{2}(\alpha+1)}{\Gamma(\alpha+1/2)\sqrt{\pi}} (uvw)^{-2\alpha}A(u,v,w)^{2\alpha-1}
\end{equation}

\noindent where  $A(u,v,w)$ is the area of  a triangle of sides $u,v,w\in\Rm$ defined by \eqref{Área-lados}.

\vspace{.2in}

For the  $n$-dimensional case, let $x,y,z\in\Rnm$ and $\mu=(\mu_1,\ldots,\mu_n)$ such that $\mu_{i}>-\frac{1}{2}$ for all  $i=1,\ldots,n$. We define

\begin{equation}\label{nDH} 
\DHI_{\mu}(x,y,z)=
\prod_{i=1}^{n} \mathfrak{D}_{\mu_i}(x,y,z)
\end{equation}
where  $\mathfrak{D}_{\mu_i}$ is given by  \eqref{DH}.

\vspace{.2in}
A convolution operation associated to the $n$-dimensional Hankel transform $\nHH_{\mu}$ can be defined. Given $f,g$ defined on $\Rnm$, the Hankel convolution associated to the transformation  $\nHH_{\mu}$ is defined formally by

\begin{equation}\label{nconv_H}
f\convH g(x)=\int_{\Rnm}\int_{\Rnm}
f(y)\:g(z)\:\DHI_{\mu}(x,y,z)\:s(y)\:s(z)\:dy\:dz
\end{equation}

\noindent where  $x,y,z\in\Rnm$.

\begin{remark}[Relation between $\DZ_{\mu}$ and $\DHI_{\mu}$]
\begin{equation}\label{}
\DZ_{\mu}(x,y,z)=
C_{\mu}^{-2}(xyz)^{1/2-\mu}\DHI_{\mu}(x,y,z)
\end{equation}
\noindent where $\DHI_{\mu}(x,y,z)$ is given by  \eqref{nDH} and $\DZ_{\mu}(x,y,z)$ is given by \eqref{nDZ}.
\end{remark}

\begin{proposition}\label{prop_nDH} In this proposition we summarize some properties for the kernel $\DHI_{\mu}(x,y,z)$ given by  \eqref{nDH}.
\begin{enumerate}
\item[(i)] $\DHI_{\mu}(x,y,z)> 0$.
\item[(ii)] $\int_{\Rnm}\DHI_{\mu}(x,y,z)
       \prod\limits_{i=1}^{n}\left\{(z_{i}t_{i})^{-\mu_{i}}
       J_{\mu_{i}}(z_{i}t_{i})
       \right\} s(z)\:dz=
       C_{\mu}\prod\limits_{i=1}^{n}\left\{
       (x_{i}t_{i})^{-\mu_{i}}J_{\mu_{i}}(x_{i}t_{i})
       \right\}
       \prod\limits_{i=1}^{n}\left\{
       (y_{i}t_{i})^{-\mu_{i}}J_{\mu_{i}}(y_{i}t_{i})
       \right\}$
\item[(iii)] $\int_{\Rnm}\DHI_{\mu}(x,y,z)\:s(z)\:dz = 1$
\end{enumerate}
\noindent where $x,y,z,t\in\Rnm$ and $J_{\mu_{i}}$ denotes the well known Bessel function of first kind and order  $\mu_{i}$ given by \eqref{Ap-bessel} for all $i=1,\ldots,n$.
\end{proposition}

\begin{proof}
The proof of  $(ii)$ It is analogous to that of the Proposition \ref{prop-nDZ}, and it will be enough to observe that
\[
\DHI_{\mu}(x,y,z)
\prod_{i=1}^{n}
\left\{(z_{i}t_{i})^{-\mu_{i}}
J_{\mu_{i}}(z_{i}t_{i})\right\}z^{2\mu+1}=
\prod_{i=1}^{n}\left\{
(z_{i}t_{i})^{-\mu_{i}} J_{\mu_{i}}(z_{i}t_{i}) \mathfrak{D}_{\mu_i}(x_i,y_i,z_i)z_{i}^{2\mu_{i}+1} \right\},
\]
\noindent is a product of functions in $z_{i}$ which are integrables in  $\Rm$.

To see $(iii)$, let us note that since $\mathfrak{D}_{\mu_i}(x_i,y_i,z_i)\in L^{1}((0,\infty), s(z_{i}))$ for all $i=1,\ldots,n\:$ then
\[\DHI_{\mu}(x,y,z)\in L^{1}(\Rnm, s(z)),\]
and
\[\left|
\DHI_{\mu}(x,y,z)\prod_{i=1}^{n}\left\{(z_{i}t_{i})^{-\mu_{i}}
J_{\mu_{i}}(z_{i}t_{i})\right\}
\right|
\leq \frac{1}{C_{\mu}}\:\DHI_{\mu}(x,y,z).\]
The result follows from  the Dominated Convergence Theorem.
\end{proof}

\begin{theorem}\label{teo1}
Let $\{\phi_{m}\}\subset L^{1}(\Rnm,s)$ a sequence of functions such that:
\begin{enumerate}
\item[(1)] $\phi_{m}(x)\geq 0$ in $\Rnm$.
\item[(2)] $\int_{\Rnm} \phi_{m}(x)\,s(x)\:dx = 1$
           for all $m\in\N$,
\item[(3)] For all  $\eta>0$,
$\lim\limits_{m\to\infty}
\int_{\|x\|>\eta} \phi_{m}(x)\,s(x)\,dx=0$.
\end{enumerate}
If $f\in L^{1}(\Rnm,s)$ then
$\lim\limits_{n\to \infty}
\|f\convH \phi_{m}-f\|_{L^{1}(\Rnm,s)}=0$.
\end{theorem}
\begin{proof}
 This result is a $n$-dimensional generalization of \cite[Corollary 2c]{Hi60}, relative to  approximate identities.
\end{proof}

\begin{lemma}\label{Cap3-Prop3}
Let $f,g$ be functions in  $L^{1}(\Rnm,s)$, then
  \begin{displaymath}
  \int_{\Rnm} \nHH_{\mu}f(t)\,g(t)\,s(t)\,dt=
  \int_{\Rnm} f(t)\,\nHH_{\mu}g(t)\,s(t)\,dt.
  \end{displaymath}
\end{lemma}

\begin{theorem}\label{TeoInv-nHH}
If $f(x)\in L^1(\Rnm,s)$ and $\nHH_{\mu} f(t)\in L^1(\Rnm,s)$ then $f(x)$ may be redefined on a set of
measure zero so that it is continuous in $x\in\Rnm$, and then

\begin{equation}\label{FormulaInv-nHH}
f(x)=\int_{\Rnm} \nHH_{\mu}f(t)
\left\{\prod_{i=1}^{n}
(x_it_i)^{-\mu_i}\,J_{\mu_i}(x_it_i)\,t_i^{2\mu_i+1}
\right\} dt
\end{equation}
for almost every  $x\in\Rnm$.
\end{theorem}

\begin{proof}

We consider the sequence $\{\phi_{m}\}_{m\in\N}$ defined by
\begin{equation}\label{kernel}
\phi_{m}(x)=m^{|\mu+1|}\,e^{-\frac{\|x\|^{2}m}{2}}.
\end{equation}
This sequence verifies conditions $(1)$, $(2)$ and $(3)$ of Theorem \ref{teo1}, then if $f\in L^{1}(\Rnm,s)$,
\[
\lim\limits_{n\to \infty}
\|f\convH \phi_{m}-f\|_{L^{1}(\Rnm,s)}=0.
\]

\noindent Let us show that:

\begin{equation}\label{TeoInv-eq1}
\phi_{m}\convH f(x)=
\int_{\Rnm} \nHH_{\mu}(f)(z)\:
e^{-\frac{\|z\|^2}{2m}}
\left\{\prod_{i=1}^{n}
(x_iz_i)^{-\mu_i}\,J_{\mu_i}(x_iz_i)
\right\}
\,z^{2\mu+1}\,dz.
\end{equation}

\noindent To see this we define
\begin{equation}\label{TeoInv-eq2}
G_{x}(z) = e^{-\frac{\|z\|^2}{2m}}
\left\{\prod_{i=1}^{n}
(x_iz_i)^{-\mu_i}\,J_{\mu_i}(x_iz_i)
\right\}.
\end{equation}

\medskip

\noindent Clearly,   $G_{x}(z) \in L^1(\Rnm,s)$ and from Lemma \ref{Cap3-Prop3} we have

\begin{align}
\int_{\Rnm} &\nHH_{\mu}(f)(z)\:
e^{-\frac{\|z\|^2}{2m}}
\left\{\prod_{i=1}^{n}
(x_iz_i)^{-\mu_i}\,J_{\mu_i}(x_iz_i)
\right\}
\,z^{2\mu+1}\,dz\nonumber\\
&=
\int_{\Rnm} \nHH_{\mu}f(z)\:
G_{x}(z)\,z^{2\mu+1}\,dz\nonumber\\
&=
\int_{\Rnm}f(t)\:
\nHH_{\mu}(G_{x}(z))(t)\,t^{2\mu+1}\,dt\label{TeoInv-eq3}
\end{align}

\noindent Moreover,
\begin{align}
\nHH_{\mu}(G_{x}(z))(t)
&=
\int_{\Rnm} G_{x}(z)
\left\{\prod_{i=1}^{n}
(z_it_i)^{-\mu}\,J_{\mu_i}(z_it_i)\,z_i^{2\mu_i+1}
\right\} dz\nonumber\\
&=
\int_{\Rnm}
e^{-\frac{\|z\|^2}{2m}}
\left\{\prod_{i=1}^{n}
(x_iz_i)^{-\mu_i}\,J_{\mu_i}(x_iz_i)
\right\}
\left\{\prod_{i=1}^{n}
(z_it_i)^{-\mu_i}\,J_{\mu_i}(z_it_i)
\right\}
\,z^{2\mu+1}\:dz\nonumber\\
&=
\int_{\Rnm}
e^{-\frac{\|z\|^2}{2m}}
\left\{
\int_{\Rnm}\DHI_{\mu}(x,t,\xi)
\left\{\prod_{i=1}^{n}
(\xi_iz_i)^{-\mu_i}\,J_{\mu_i}(\xi_iz_i)
\right\} s(\xi)\:d\xi\right\}
s(z)\:dz. \label{TeoInv-eq4}
\end{align}

\noindent Since
\[
\int_{\Rnm}
e^{-\frac{\|z\|^2}{2m}}
\left\{
\int_{\Rnm}\DHI_{\mu}(x,t,\xi)
\left\{\prod_{i=1}^{n}
|(\xi_iz_i)^{-\mu_i}\,J_{\mu_i}(\xi_iz_i)|
\right\} s(\xi)\:d\xi\right\}
s(z)\:dz
\]

\[
\leq
C_{\mu}^{-1} \int_{\Rnm}
e^{-\frac{\|z\|^2}{2m}}
\left\{
\int_{\Rnm}\DHI_{\mu}(x,t,\xi)\:s(\xi)\:d\xi\right\}
s(z)\:dz<\infty
\]

\medskip
\noindent it is possible the change the order of integration in \eqref{TeoInv-eq4}, then

\clearpage
\begin{align*}
\nHH_{\mu}(G_{x}(z))(t)
&=
\int_{\Rnm}
\left\{
\int_{\Rnm}
e^{-\frac{\|z\|^2}{2m}}
\left\{\prod_{i=1}^{n}
(\xi_iz_i)^{-\mu_i}\,J_{\mu_i}(\xi_iz_i)
\right\} s(z)\:dz\right\}
\DHI_{\mu}(x,t,\xi)
s(\xi)\:d\xi\\
&=
C_{\mu}^{-2}
\int_{\Rnm}
\left\{
\int_{\Rnm}
e^{-\frac{\|z\|^2}{2m}}
\left\{\prod_{i=1}^{n}
J_{\mu_i}(\xi_iz_i)\:z_{i}^{\mu_i+1}
\right\} \:dz\right\}
\DHI_{\mu}(x,t,\xi)\:
\xi^{\mu+1}\:d\xi\\
&=
C_{\mu}^{-2}
\int_{\Rnm}
\xi^{\mu}\:m^{|\mu+1|}\:e^{-\frac{\|\xi\|^{2}m}{2}}
\DHI_{\mu}(x,t,\xi)\:
\xi^{\mu+1}\:d\xi\\
&=
C_{\mu}^{-1}
\int_{\Rnm}
m^{|\mu+1|}\:e^{-\frac{\|\xi\|^{2}m}{2}}
\DHI_{\mu}(x,t,\xi)\:
s(\xi)\:d\xi\\
\end{align*}

\noindent where we have used \eqref{Prop3.14-eq1} with $a=1/m$. From the last equality we obtain that

\begin{align}
\int_{\Rnm} &
\nHH_{\mu}(G_{x}(z))(t)\,f(t)\,t^{2\mu+1}\,dt\nonumber\\
&=
\int_{\Rnm}\left\{
C_{\mu}^{-1}
\int_{\Rnm}
m^{|\mu+1|}\:e^{-\frac{\|\xi\|^{2}m}{2}}
\DHI_{\mu}(x,t,\xi)\:
s(\xi)\:d\xi\right\}
f(t)\,t^{2\mu+1}\,dt\nonumber\\
&=
\int_{\Rnm}\left\{
\int_{\Rnm}
m^{|\mu+1|}\:e^{-\frac{\|\xi\|^{2}m}{2}}
\DHI_{\mu}(x,t,\xi)\:
s(\xi)\:d\xi\right\}
f(t)\,s(t)\,dt\nonumber\\
&=\phi_{m}\convH f(x)\label{TeoInv-eq5}.
\end{align}

\noindent From \eqref{TeoInv-eq3} and  \eqref{TeoInv-eq5} to obtain \eqref{TeoInv-eq1}. We may now take limit in \eqref{TeoInv-eq1}, and considering that $\nHH_{\mu}f \in L^{1}(\Rnm,s)$ and

\begin{equation*}
\left|
e^{-\frac{\|z\|^2}{2m}}
\left\{\prod_{i=1}^{n}
(z_ix_i)^{-\mu_i}\,J_{\mu_i}(z_ix_i)\right\}
\nHH_{\mu}(f)(z)
\right|
\leq
C_{\mu}^{-1}\,e^{-\frac{\|z\|^2}{2m}}\,\left|\nHH_{\mu}(f)(z)\right|
\leq
C\,\left|\nHH_{\mu}(f)(z)\right|,
\end{equation*}

\noindent
then we obtain in the right side of \eqref{TeoInv-eq1} by the dominated convergence theorem that
\[
\lim\limits_{m\to\infty} \int_{\Rnm} \nHH_{\mu}(f)(z)\,
e^{-\frac{\|z\|^2}{2m}}
\left\{\prod_{i=1}^{n}
(z_ix_i)^{-\mu_i}\,J_{\mu_i}(z_ix_i)\right\}
\,z^{2\mu+1}\,dz=
\]

\begin{equation}\label{TeoInv-eq8}
\int_{\Rnm} \nHH_{\mu}(f)(z)\,
\left\{\prod_{i=1}^{n}
(z_ix_i)^{-\mu_i}\,J_{\mu_i}(z_ix_i)\right\}
\,z^{2\mu+1}\,dz=
\nHH_{\mu}(\nHH_{\mu}f)(x).
\end{equation}

\noindent So,
\begin{equation}
\lim\limits_{m\to\infty}\phi_{m}\#f(x)=
\nHH_{\mu}(\nHH_{\mu}f)(x).
\end{equation}

\noindent On the other hand, by \eqref{teo1} there exists a subsequence $\{\phi_{m_{k}}\convH f\}_{k \in \N}$ such that
\begin{equation}
\lim\limits_{k\to\infty}
\phi_{m_k}\convH f(x)=f(x)
\end{equation}
for almost every $x\in\Rnm$, and so this completes the proof.
\end{proof}

\medskip
From Theorem \ref{TeoInv-nHH} we can obtain the inversion theorem  for the Hankel transform $\nHZ_{\mu}$ given by \eqref{nHZ}

\begin{proof}[Proof of Theorem \ref{TeoInv-nHZ}]\label{Proof-TeoInv-nHZ}
If $f\in L^{1}(\Rnm,x^{\mu+1/2})$ then  $x^{-\mu-1/2}f\in L^{1}(\Rnm,s)$.

\medskip
\noindent Since
\[x^{-\mu-1/2}\nHZ_{\mu}(f) = \nHH_{\mu}(x^{-\mu-1/2}f)\]
and for hypothesis
\[\nHZ_{\mu}f\in L^{1}(\Rnm,x^{\mu+1/2})\]
then $\nHH_{\mu}(x^{-\mu-1/2}f)\in L^{1}(\Rnm,s)$.

\medskip
\noindent Then the result continues to apply the Theorem \ref{TeoInv-nHH} to $x^{-\mu-1/2}f$ and obtain  \eqref{FormulaInv-nHZ}.
\end{proof}

\begin{lemma}\label{Cap3-Prop3 para nHZ}
Let $f,g$ be functions in  $L^{1}(\Rnm,sr)$, then
  \begin{displaymath}
  \int_{\Rnm} \nHZ_{\mu}f(t)\,g(t)\,dt=
  \int_{\Rnm} f(t)\,\nHZ_{\mu}g(t)\,dt.
  \end{displaymath}
\end{lemma}

\begin{proof}[Proof of Lema \ref{lema3.11}]\label{proof-lema3.11} Let $\phi\in\mathcal{H}_{\mu}$.

\begin{proof}[(i)]
It follows easily from induction on $m$.
\phantom{\qedhere}
\end{proof}

\begin{proof}[(ii)]
The proof of this result follows for induction on $m$. First let us observe that
\begin{align*}
(\nHZ_{\mu}(\|y\|^{2}+\lambda)^{-1}\nHZ_{\mu})&(-\nBZ+\lambda)\phi
=
\nHZ_{\mu}(\|y\|^{2}+\lambda)^{-1}[-\nHZ_{\mu}\nBZ\phi+\lambda\nHZ_{\mu}\phi]\\
&=
\nHZ_{\mu}(\|y\|^{2}+\lambda)^{-1}[-(-\|y\|^{2}\nHZ_{\mu}\phi)+\lambda\nHZ_{\mu}\phi]\\
&=
\nHZ_{\mu}(\|y\|^{2}+\lambda)^{-1}(\|y\|^{2}+\lambda)\nHZ_{\mu}\phi\\
&=
\nHZ_{\mu}(\nHZ_{\mu}\phi)=\phi
\end{align*}
So,
\[(\nHZ_{\mu}(\|y\|^{2}+\lambda)^{-1}\nHZ_{\mu})(-\nBZ+\lambda)=Id.\]
On the other hand
\begin{align*}
(-\nBZ+\lambda)&(\nHZ_{\mu}(\|y\|^{2}+\lambda)^{-1}\nHZ_{\mu})\phi
=
-\nBZ\nHZ_{\mu}(\|y\|^{2}+\lambda)^{-1}\nHZ_{\mu}\phi+
\lambda(\nHZ_{\mu}(\|y\|^{2}+\lambda)^{-1}\nHZ_{\mu})\phi\\
&=
-\nHZ_{\mu}[-\|y\|^{2}(\|y\|^{2}+\lambda)^{-1}\nHZ_{\mu}\phi]+
\nHZ_{\mu}[\lambda(\|y\|^{2}+\lambda)^{-1}\nHZ_{\mu}\phi]\\
&=
\nHZ_{\mu}[\|y\|^{2}(\|y\|^{2}+\lambda)^{-1}\nHZ_{\mu}\phi+\lambda(\|y\|^{2}+\lambda)^{-1}\nHZ_{\mu}\phi]\\
&=
\nHZ_{\mu}(\|y\|^{2}+\lambda)(\|y\|^{2}+\lambda)^{-1}\nHZ_{\mu}\phi\\
&=
\nHZ_{\mu}(\nHZ_{\mu}\phi)=\phi
\end{align*}

\noindent So,
\[\nHZ_{\mu}(\|y\|^{2}+\lambda)^{-1}\nHZ_{\mu}=(-\nBZ+\lambda)^{-1},\]
then
\[\nHZ_{\mu}(-\nBZ+\lambda)^{-1}=
(\|y\|^{2}+\lambda)^{-1}\nHZ_{\mu}.\]

Now, suppose that $\nHZ_{\mu} (-\nBZ+\lambda)^{-m} = (\lambda+\|y\|^{2})^{-m}\nHZ_{\mu}$,

\begin{align*}
\nHZ_{\mu}(-\nBZ+\lambda)^{-(m+1)} \phi
&=
\nHZ_{\mu}(-\nBZ+\lambda)^{-m}(-\nBZ+\lambda)^{-1}\phi\\
&=
(\|y\|^{2}+\lambda)^{-m}\nHZ_{\mu}(-\nBZ+\lambda)^{-1}\phi\\
&=
(\|y\|^{2}+\lambda)^{-m}(\|y\|^{2}+\lambda)^{-1}\nHZ_{\mu}\phi\\
&=
(\|y\|^{2}+\lambda)^{-(m+1)}\nHZ_{\mu}\phi
\end{align*}
\phantom{\qedhere}
\end{proof}

\begin{proof}[(iii)]
 By induction on $m$,
\begin{align*}
\nHZ_{\mu}[-\nBZ(-\nBZ+\lambda)^{-1}]\phi
&=
-\nHZ_{\mu}\nBZ[(-\nBZ+\lambda)^{-1}\phi]\\
&=
(-1)[-\|y\|^{2}\nHZ_{\mu}(-\nBZ+\lambda)^{-1}\phi]\\
&=
\|y\|^{2}\nHZ_{\mu}(-\nBZ+\lambda)^{-1}\phi\\
&=
\|y\|^{2}(\|y\|^{2}+\lambda)^{-1}\nHZ_{\mu}\phi
\end{align*}

\noindent Suppose that  $\nHZ_{\mu}(-\nBZ(-\nBZ+\lambda)^{-1})^{m}=
\|y\|^{2m}(\lambda+\|y\|^{2})^{-m}\nHZ_{\mu}$.
\begin{align*}
\nHZ_{\mu}&[-\nBZ(-\nBZ+\lambda)^{-1}]^{m+1}\phi
=
\nHZ_{\mu}[-\nBZ(-\nBZ+\lambda)^{-1}][-\nBZ(-\nBZ+\lambda)^{-1}]^{m}\phi\\
&=
\|y\|^{2}(\|y\|^{2}+\lambda)^{-1}
\nHZ_{\mu}([-\nBZ(-\nBZ+\lambda)^{-1}]^{m}\phi)\\
&=\|y\|^{2}(\|y\|^{2}+\lambda)^{-1}
\|y\|^{2m}(\|y\|^{2}+\lambda)^{-m}\nHZ_{\mu}\phi\\
&=\|y\|^{2(m+1)}(\|y\|^{2}+\lambda)^{-(m+1)}\nHZ_{\mu}\phi
\end{align*}
\phantom{\qedhere}
\end{proof}

\noindent The equalities for the case $u\in\mathcal{H'}_{\mu}$ are followed by transposition.
\end{proof}

\section{Similarity of Bessel operators}
\begin{remark}\label{proof similaridad Op de Bessel}
The operators $\nBH$ and $\nBZ$ given by \eqref{nOBH} and \eqref{nOBZ} respectively are related through
\begin{equation}
\nBZ=x^{\mu+1/2}\nBH x^{-\mu-1/2} .
\end{equation}
\end{remark}

\bigskip
\noindent Let $\phi\in C^{2}(\Rnm)$, then

\begin{align*}
&x^{-\mu-1/2}\nBZ x^{\mu+1/2}\phi(x)=
x^{-\mu-1/2}\left(
\sum_{i=1}^{n}
\frac{\partial^{2}}{\partial x_{i}^{2}}-
\frac{4\mu_{i}^{2}-1}{4x_{i}^{2}}
\right) x^{\mu+1/2} \phi(x)\\
&=
\sum_{i=1}^{n}
x^{-\mu-1/2}\frac{\partial^{2}}{\partial x_{i}^{2}} x^{\mu+1/2}\phi(x)-
\sum_{i=1}^{n}
\frac{4\mu_{i}^{2}-1}{4x_{i}^{2}}\phi(x)
\end{align*}

\noindent and also

\begin{equation*}
\begin{split}
&x^{-\mu-1/2}\frac{\partial^{2}}{\partial x_{i}^{2}}
\{ x^{\mu+1/2}\phi(x)\}
= x^{-\mu-1/2}x^{\mu+1/2}x_{i}^{-\mu_{i}+1/2}\frac{\partial^{2}}{\partial x_{i}^{2}}
\{x_{i}^{\mu_{i}+1/2}\phi(x)\}
=x_{i}^{-\mu_{i}-1/2}\frac{\partial^{2}}{\partial x_{i}^{2}}
\{ x_{i}^{\mu_{i}+1/2}\phi(x)\}\\
&\quad = x_{i}^{-\mu_{i}-1/2}\frac{\partial}{\partial x_{i}}
\{(\mu_{i}+1/2)\: x_{i}^{\mu_{i}-1/2}\phi(x)+
x_{i}^{\mu_{i}+1/2}\frac{\partial}{\partial x_{i}}\phi(x)\}\\
&\quad = x_{i}^{-\mu_{i}-1/2}
\{(\mu_{i}+1/2)(\mu_{i}-1/2)\: x_{i}^{\mu_{i}-3/2}\phi(x)+
(\mu_{i}+1/2)\: x_{i}^{\mu_{i}-1/2}\frac{\partial}{\partial x_{i}}\phi(x)+\\
&
\qquad+(\mu_{i}+1/2)\: x_{i}^{\mu_{i}-1/2}\frac{\partial}{\partial x_{i}}\phi(x)+
x_{i}^{\mu_{i}+1/2}\frac{\partial^{2}}{\partial x_{i}^{2}}\phi(x)
\}\\
&\quad =(\mu_{i}^{2}-1/4)\:x_{i}^{-2}\phi(x)+
2(\mu_{i}+1/2)\:x_{i}^{-1}\frac{\partial}{\partial x_{i}}\phi(x)+
\frac{\partial^{2}}{\partial x_{i}^{2}}\phi(x),
\end{split}
\end{equation*}

\noindent from where

\begin{align*}
&x^{-\mu-1/2}\nBZ x^{\mu+1/2}\phi(x)=
\sum_{i=1}^{n}
x^{-\mu-1/2}\frac{\partial^{2}}{\partial x_{i}^{2}} x^{\mu+1/2}\phi(x)-
\sum_{i=1}^{n}
\frac{4\mu_{i}^{2}-1}{4x_{i}^{2}}\phi(x)\\
& \quad =
\sum_{i=1}^{n}
(\mu_{i}^{2}-1/4)\:x_{i}^{-2}\phi(x)+
2(\mu_{i}+1/2)\:x_{i}^{-1}\frac{\partial}{\partial x_{i}}\phi(x)+
\frac{\partial^{2}}{\partial x_{i}^{2}}\phi(x)-
\sum_{i=1}^{n}
\frac{4\mu_{i}^{2}-1}{4x_{i}^{2}}\phi(x)\\
& \quad =
\sum_{i=1}^{n}
2(\mu_{i}+1/2)\:x_{i}^{-1}\frac{\partial}{\partial x_{i}}\phi(x)+
\frac{\partial^{2}}{\partial x_{i}^{2}}\phi(x)=\nBH\phi(x)
\end{align*}


\begin{thebibliography}{9}
%
\bibitem{Ba60} A. V. Balakrishnan, Fractional powers of closed operators and the semi-groups generated by them, Pacific J. Math., 10 (1960), pp 419-437.

\bibitem{BKN02} K. Bogdan, T. Kulczycki, A. Nowak, Gradient estimates for harmonic and $q$-harmonic functions of symmetric stable processes, Illinois J. Math., 48 (2002), pp. 541-556.

\bibitem{CDAL14} W. Chen, L. D’Ambrosio, Y. Li, Some Liouville theorems for the fractional Laplacian, Nonlinear Anal. Theor., 121 (2015), pp. 370-381.

\bibitem{Delsarte} J. Delsarte, Sur une extension de Ia formule de Taylor, Journal de Mathématiques Pures et Appliquées. Neuvième Série, 17 (1938).

\bibitem{F099} G. B: Folland, Real Analysis: Modern Techniques and Their Applications. 2nd ed., A Wiley-Interscience publication, New York, 1999.

\bibitem{GMQ18} V. Galli, S. Molina, A. Quintero, A Liouville theorem for some Bessel generalized operators, Integral Transforms and Special Functions, 29 (2018), pp 367-383.

\bibitem{Hi60} I. I. Hirschman Jr., Variation Diminishing Hankel Transforms, J. Analyse Math., 8 (1960/61), pp. 307-336.

\bibitem{Li06} Y. Li, Local gradient estimates of solutions to some conformally invariant fully nonlinear equations, C. R. Math.  Acad. Sci. Paris, Ser. I 343(2006) pp. 249-252.

\bibitem{MS01} C. Mart\'inez Carracedo, M. Sanz Alix, The theory of fractional powers of operators, vol 187, Elsevier Science, 2001.

\bibitem{Mo03} S. Molina, A Generalization of the spaces $\mathcal{H}_\mu$, $\mathcal{H}'_\mu$ and the space of Multipliers, Actas del VII Congreso Dr. Antonio A. R. Monteiro, (2003), pp. 49-56.

\bibitem{Mo18} S. Molina, Distributional fractional powers of similar operators. Applications to the Bessel operators, Commun. Korean Math. Soc., 33 (2018), pp. 1249-1269.
    
\bibitem{MT07} S. Molina, S. E. Trione, $n-$Dimensional Hankel transform and complex powers of Bessel operator, Integr. Transf. Spec. F., 18 (2007), pp. 897-911.

\bibitem{MT08} S. Molina, S. E. Trione, On the $n-$dimensional Hankel transforms of arbitrary order, Integr. Transf. Spec. F., 19 (2008) pp.327-332.

\bibitem{Ob} Oberhettinger, F, Tables of Bessel transforms, pringer Verlag, New York 1972.

\bibitem{Schaefer} H. H. Schaeffer, Topological Vector Spaces., Springer-Verlag, 1971.

\bibitem{Wa} G. N. Watson, A Treatise on the Theory of Bessel Functions. 2nd ed., Cambridge University Press, 1995.

\bibitem{Ze66} A. H. Zemanian, A distributional Hankel transformation, SIAM J. Appl. Math., 14 (1966) pp. 561-576.

\bibitem{Ze87} A. H. Zemanian, Generalized Integral Transformations, Dover Publications, New York, 1987.

\bibitem{ZCCY14} R. Zhuo, W. Chen, X. Cui, Z. Yuan,A Liouville Theorem for the Fractional Laplacian, 2014, arXiv:1401.7402.

\end{thebibliography}

\end{document}